\documentclass[12pt]{amsart}
\usepackage[T1]{fontenc}
\usepackage[utf8]{inputenc}
\usepackage[margin=3cm]{geometry}
\usepackage{graphicx,xcolor}
\usepackage{amsfonts,amsmath,amssymb,dsfont,mathrsfs,xfrac}
\usepackage{enumitem,interval,suffix,xparse}
\usepackage[colorlinks=true,linkcolor=purple,citecolor=teal,
hypertexnames=false]{hyperref}
\usepackage{booktabs}
\usepackage[justification=justified]{caption}
\usepackage{ragged2e}

\let\originalleft\left
\let\originalright\right
\renewcommand\left{\mathopen{}\mathclose\bgroup\originalleft}
\renewcommand\right{\aftergroup\egroup\originalright}

\usepackage{nowidow}
\setnowidow[4]
\setnoclub[4]
\usepackage[giveninits=true,isbn=false,style=numeric,sortcites=true,url=false]{biblatex}
\newtheorem{theorem}{Theorem}[section]
\newtheorem{lemma}[theorem]{Lemma}
\newtheorem{remark}[theorem]{Remark}

\newtheorem{proposition}[theorem]{Proposition}

\newtheorem{corollary}[theorem]{Corollary}

\newtheorem*{GVC}{Generalized variance conjecture}

\newcommand\numstyle{\mathbb}
\makeatletter
\@for\@t:=N,Z,Q,R,C,H,F\do{\expandafter\edef\csname\@t\endcsname{\noexpand\numstyle\@t}\expandafter\edef\csname\@t\@t\endcsname{\noexpand\numstyle\@t}}
\@for\@t:=A,B,C,D,E,F,G,H,I,J,K,L,M,N,O,P,Q,R,S,T,U,V,W,X,Y,Z\do{\expandafter\edef\csname d\@t\endcsname{\noexpand\mathbb\@t}\expandafter\edef\csname c\@t\endcsname{\noexpand\mathcal\@t}\expandafter\edef\csname r\@t\endcsname{\noexpand\mathrm\@t}}
\makeatother
\WithSuffix\newcommand\RR+{\RR^{\scalebox{0.75}{$+$}}}
\newcommand*\ceq{\mathrel{\vcenter{\baselineskip0.5ex \lineskiplimit0pt
			\hbox{\scriptsize.}\hbox{\scriptsize.}}}=}
\newcommand*\cequiv{\mathrel{\vcenter{\baselineskip0.5ex \lineskiplimit0pt
			\hbox{\scriptsize.}\hbox{\scriptsize.}}}\equiv}

\newcommand*\colspace{\mathrel{\vcenter{\baselineskip0.5ex \lineskiplimit0pt
			\hbox{\scriptsize\phantom{.}}\hbox{\scriptsize\phantom{.}}}}}

\intervalconfig{left open fence=(, right open fence=)}
\newcommand\newIntval[2]{\expandafter\DeclareDocumentCommand\csname #1\endcsname{ o m m }{\IfValueTF{##1}{\interval[#2,##1]{##2}{##3}}{\interval[#2]{##2}{##3}}}}
\newIntval{oo}{open}
\newIntval{oc}{open left}
\newIntval{co}{open right}
\newIntval{cc}{}
\newcommand\bs[1]{\boldsymbol{\mathbf{#1}}}
\newcommand\eps\varepsilon
\newcommand\dd{\mathrm d}

\newcommand\ee{\operatorname{e}}

\newcommand\ind[1]{\mathds{1}_{#1}}
\newcommand\wt\widetilde
\newcommand\mc\mathcal
\DeclareMathOperator\E{\mathbb E}
\DeclareMathOperator\EE{\mathbb E}

\DeclareMathOperator\PP{\mathbb P}
\newcommand\prb[2][]{\PP_{#1}\left(#2\right)}
\newcommand\esp[2][]{ \EE{#2}}
\newcommand\Esp[1]{ \EE{#1}}

\let\Enp\EE
\DeclareMathOperator\Var{Var}

\newcommand\wtZ{Z}
\let\ln\log
\newcommand\subref[2]{\hyperref[#1]{\ref*{#1}.\ref*{#1.#2}}}

\newcommand\UB[1][p]{B(S_#1^n)}
\newcommand\UBE{B_E(S_p^n)}
\newcommand\HS[1]{\|#1\|_{\text{\tiny\upshape\rmfamily HS}}}

\usepackage{mathrsfs}
\DeclareMathAlphabet{\mathpzc}{OT1}{pzc}{m}{it}

\newcommand\mail[1]{\normalfont\href{mailto:#1}{\texttt{#1}}}
\newcommand\univaddress{Univ Gustave Eiffel, Univ Paris Est Creteil, CNRS, LAMA UMR8050, F-77447 Marne-la-Vallée, France
	}

\title[inertia moments and variance conjecture in Schatten balls]{\bfseries 
Asymptotics of the inertia moments and the variance conjecture in Schatten balls}
\date{\today}

\author{B.\ Dadoun, M.\ Fradelizi, O.\ Guédon, \and P.-A.\ Zitt}
\thanks{B.\ Dadoun acknowledges support from the ANR-17-CE40-0017 ASPAG}
\begin{document}

\begin{abstract}
We study the first and second orders of the asymptotic expansion, as the dimension goes to infinity, of the moments of the Hilbert-Schmidt norm of a uniformly distributed matrix in the $p$-Schatten unit ball. We consider the case of matrices with real, complex or quaternionic entries, self-adjoint or not. 
When $p>3$, this asymptotic expansion allows us 
to establish a generalized version of the variance conjecture for the family of $p$-Schatten unit balls of self-adjoint matrices.
\end{abstract}
\subjclass{52A23, 46B07, 46B09, 60B20, 82B05}
\keywords{Asymptotic convex geometry; Schatten balls; inertia moments;
	variance conjecture}

\maketitle

\section{Introduction}

Let $\F\ceq \R$, $\C$, or the quaternionic field~$\HH$. For $n \ge 1$, we work
on~$\F^n$ which is seen as a $\beta n$-dimensional vector space over~$\R$,
with $\beta \in \{1,2,4\}$. We generically denote by~$E$
either the space $\mathcal M_n(\F)$
of $n\times n$ matrices with entries from the field~$\F$, or the
subspace  of self-adjoint matrices in $\mathcal{M}_n(\F)$.
We view~$E$ as a vector space over~$\R$, whose dimension~$d_n$ depends
on~$n$, $\F$, and whether or not we impose self-adjointness (see~\eqref{eq:dimension} below for the exact formula). 
In any case, we equip~$E$ with the Euclidean structure defined by
$ {\HS T^2}\ceq{\rm tr}{(T^*T)}$, and with the Lebesgue measure
denoted by~$\dd T$.
We write~$|A|$ for the Lebesgue measure of any Borel set $A
\subseteq E$, when it is finite.

For every $z\in \F^n$, let
\[
\|z\|_p  \ceq \left(\sum_{i=1}^n\lvert z_i\rvert^p \right)^{\!\frac1p}
\quad\text{for}\enspace
1\leq p <\infty,\quad\text{and}\enspace
\|z\|_{\infty} \ceq\max_{1\le i\leq n}|z_i|.
\]
For any matrix $T \in \mathcal M_n(\F)$, let
$s(T) \ceq (s_1(T),\dots, s_n(T)) \in \R^n$ be  the set of singular values of~$T$, i.e., the eigenvalues of $\sqrt{T^{*} T}$. 
For every $ 1\leq p \leq \infty$, let $\sigma_p (T) = \| s(T)
\|_p $: this defines a norm on $\mathcal{M}_n(\F)$, called the
$p$-Schatten norm (see \cite[Chapter IV]{Bhatia97}), and the
corresponding normed space is denoted by~$S_p^n$, with unit ball
\[
\UB \ceq \{ T\in \mathcal M_n(\dF): \sigma_p (T)\leq 1\}.
\]
In the special case $p=2$,  $\sigma_2(T) ={\HS T}$ and we recover the Euclidean structure.

Let us note that, when~$E$ is a space of self-adjoint matrices, all
matrices in~$E$ are diagonalizable, and $\sigma_p(T)$ is then the
$p$-norm of the vector $(\lambda_1(T),\dots, \lambda_n(T))$ of  eigenvalues of~$T$.
Taking random Gaussian entries for~$T$ then leads to
the classical orthogonal (GOE), unitary (GUE) and symplectic (GSE) ensembles in
random matrix theory, see for instance \cite[Chapter~2]{Mehta04}. 

The purpose of this paper is to study the first and second orders in the asymptotic expansion, as~$n$
goes to infinity, of the $q$-inertia moment
\begin{equation}\label{eq:def_Iq}
I_q(\UBE) \ceq
\frac{\left(\esp{\HS T^q}\right)^{\sfrac1q}}{\lvert\UBE\rvert^{\sfrac1{d_n}} },
\end{equation}
for~$T$  uniformly distributed in $\UBE\ceq B(S_p^n)\cap E$.
These quantities appear in various conjectures in asymptotic geometric
analysis that we now present.

\subsection{Relevant conjectures in asymptotic geometric analysis.}

The interplay between the classical conjectures in asymptotic
geometric analysis that are the hyperplane conjecture, the
Kannan-Lov\'{a}sz-Simonovits
conjecture and the variance conjecture, is quite intricate; we state
them here briefly to provide context and  refer the interested reader to~\cite{Alonso15, Brazitikos14, Lee19, Klartag21} for
an in-depth presentation of the field.

Let~$K$ be a symmetric convex body in~$\R^d$. The isotropic constant of~$K$ is defined by
\begin{equation}\label{isotropy}
	L_{K}  \:\ceq \min_{\det(T)=1}  \frac{1}{\sqrt{d}}\left(\frac{1}{|K|^{1+{\frac 2d}}}\int_K \| T\bs x\|_2^{2}\,\dd\bs x\right)^{\!\frac12},
\end{equation}
and its covariance matrix~$\Sigma$ is defined by
\[
\Sigma_{i,j}\ceq\frac{1}{|K|} \int_K x_i x_j\, \dd\bs x,\quad 1\le i,j\le d.
\]
We say that~$K$ is in isotropic position when $\Sigma = {\rm Id}$,
in which case $L_K = \frac{1}{|K|^{1/d}}$ and the minimum in Equation~\eqref{isotropy} is attained for $T={\rm Id}$.  
The hyperplane conjecture asks for a uniform upper bound on~$L_K$ for all symmetric convex bodies and dimensions~$d$, see~\cite{Milman89, Klartag21}. 
The Kannan-Lov\'asz-Simonovits (KLS) conjecture~\cite{Kannan95} asks for the existence of a universal constant~$C$ such that, for every dimension~$d$, every random vector~$X$ uniformly distributed on a symmetric convex body $K\subset\RR^d$, and every smooth function~$f$ on~$\R^d$,
\begin{equation}
\label{eq:KLSconj}
\Var f(X)\le C \max_{\theta \in S^{d-1}} \E \langle X, \theta \rangle^2 \cdot  \esp{\|\nabla f(X)\|_2^2}.
\end{equation}
This conjecture is satisfied for various families of convex bodies~\cite{Alonso15}, like the balls~$B_p^n$~\cite{Sodin08}, the Orlicz
balls \cite{Kolesnikov18,Barthe21}.
In a recent breakthrough, Yuansi Chen~\cite{Chen21}, using methods introduced by Eldan  in~\cite{Eldan13} and developed in~\cite{Lee17}, proved that inequality~\eqref{eq:KLSconj} is valid for a value~$C$ still depending on the dimension but smaller than any power of~$d$, $C = d^{o_d(1)}$.  Specifying the KLS conjecture to $f(x)=\|x\|_2^2$ one gets the following weaker conjecture.
\begin{GVC}
 There is a constant $C \ge 1$ such that for every dimension~$d$ and for every symmetric convex body $K \subset \R^d$, we have 
\begin{equation}
\label{eq:varianceconj}
\Var \|X\|_2^2
\le C\max_{\theta \in S^{d-1}} \E \langle X, \theta \rangle^2 \cdot  \esp{\|X\|_2^2}
=
C \|\Sigma\| \cdot {\rm tr} (\Sigma)
\end{equation}
where~$X$ is uniformly distributed on~$K$ and~$\Sigma$ is its covariance matrix.
\end{GVC}

Klartag~\cite{Klartag09} proved this conjecture for every unconditional convex bodies, that means invariant under coordinate hyperplane reflections, while Barthe and Cordero~\cite{Barthe13} studied these questions for convex bodies with more general symmetries. 

These conjectures were studied in the particular case of the unit balls of the $p$-Schatten normed spaces. 
K\"onig, Meyer and Pajor~\cite{Koenig98} established the hyperplane conjecture for these families of particular bodies. Gu\'edon and Paouris~\cite{Guedon07} proved a concentration of the volume through the study of the moments of the Hilbert-Schmidt norm of a random matrix uniformly distributed in~$\UBE$. This method was generalized by Radke and Vritsiou~\cite{Radke20} and Vritsiou~\cite{Vritsiou18} to prove the generalized variance conjecture when the space of matrices is equipped with the operator norm, that is the case $p = \infty$.

\subsection{Results.}

We start by stating a result on the first order in the asymptotic expansion of $I_q(\UBE)$. We give the exact computation of the limit of $I_q(\UBE)$, correctly normalized by the square root of the dimension of~$E$ as a vector space over~$\R$. This dimension will be denoted by~$d_n$.
\begin{theorem}[Limit of the normalized inertia]
  \label{thm:limit-isotropic-constant}
Let~$E$ be $\mathcal M_n(\F)$ or
its restriction to self-adjoint matrices,
equipped with the $p$-Schatten norm for some $p \in\co{1}{\infty}$. Then for every $q >0$, 
\begin{equation}\label{eq:isotropie}
\frac{I_q(\UBE)}{\sqrt{d_n}}=
\frac{\left(\esp{{\HS T^q}}\right)^{\sfrac1q}}{\sqrt{d_n}\,\lvert\UBE\rvert^{\sfrac{1}{d_n}} }
 \:\xrightarrow[n\to\infty]{}\: \ee^{\frac{1}{2p}-\frac34}\sqrt{\frac{p}{\pi(p+2)}},
\end{equation}
where~$T$ is uniformly distributed in $\UBE$. 
\end{theorem}
In the self-adjoint case, we note that this result can be established as a consequence of a weak law of large numbers proved by Kabluchko, Prochno and Th\"ale, \cite[Theorem~4.7]{Kabluchko20a} combined with a truncation argument \cite[Theorem 11.1.2]{Pastur11}. 

The general case is of particular  interest, as it allows for instance,
for $q=2$ and $E = \mathcal{M}_n(\F)$, 
a computation 
of the limit of the isotropic constant of the unit ball $S_p^n$.
Indeed, in this (and only in this) non self-adjoint case, the Schatten unit balls are in isotropic position~\cite{Koenig98,Aubrun06,Radke20}, so that the left-hand side of~\eqref{eq:isotropie} is the isotropic constant defined by~\eqref{isotropy}, and Theorem~\ref{thm:limit-isotropic-constant} gives the precise value of its limit as~$n$ goes to infinity.

Our main result establishes, in the self-adjoint case,
the second term in the asymptotic expansion of the quotient $I_q(\UBE)/I_2(\UBE)$. 
\begin{theorem}[Asymptotic expansion of the $q$-inertia moment]\label{thm:moment-q}
Let~$E$ be the space of real symmetric, complex Hermitian   or  Hermitian quaternionic  matrices equipped with the $p$-Schatten norm for some $p \in \oo{3}{\infty}$.  Then for every $q>0$,
\[
\frac{I_q(\UBE)}{I_2(\UBE)} =1+\frac{(q-2)(p-2)^2}{16p^2\,d_n} +o\left(\frac1{d_n}\right)
\]
as~$n$ goes to infinity,
where~$I_q$ is defined in~\eqref{eq:def_Iq}. 
\end{theorem}
As a corollary, we give a positive answer to the generalized variance conjecture for these families of convex bodies.  
\begin{corollary}[Variance conjecture in Schatten balls]\label{thm:variance}
Let  $p \in \oo{3}{\infty}$ and~$E$ be the space of real sym\-metric, complex Hermitian   or  Hermitian quaternionic  matrices equipped with the Schatten $p$-norm. Then, for all $n$ large enough,
 \[
\Var ({\HS T^2})
\le \frac{1}{2}\,\|\Sigma\| \cdot {\rm tr} (\Sigma),
 \]
 where~$T$ is uniformly distributed in $\UBE$ and~$\Sigma$ is the covariance matrix of~$\UBE$.
\end{corollary}
\noindent
Indeed, specializing Theorem~\ref{thm:moment-q} to $q=4$, we get
\[
\lim_{n \to \infty} d_n \frac{\Var ({\HS T^2})}{\left(\esp{\HS T^2}\right)^2}
= \frac{(p-2)^2}{2p^2} < \frac{1}{2}.
\]
We conclude using that $\esp{\HS T^2} = {\rm tr} (\Sigma) \le d_n \, \| \Sigma \|$.

\subsection{Strategy of proof.}
Recall that~$E$ is either $\mathcal{M}_n(\F)$ or the subspace of
self-adjoint matrices in $\mathcal{M}_n(\F)$. 
Using a
strategy developed by Saint-Raymond~\cite{SaintRaymond84},
K\"onig, Meyer and Pajor~\cite{Koenig98}, 
and pushed further by Gu\'edon and Paouris~\cite{Guedon07}, Radke and Vritsiou~\cite{Radke20} and Kabluchko, Prochno and Th\"ale~\cite{Kabluchko20a}, 
we first  reduce the computations of $I_q(\UBE)$
to integrals over~$\R^n$ with respect to a Gibbs probability measure~$\PP_{n,p}$.  This distribution governs a so called $\beta$-ensemble consisting of~$n$ unit charges interacting through a logarithmic Coulomb
potential and confined with an external potential, see~\eqref{eq:density}.  
Kabluchko, Prochno and Th\"ale \cite{Kabluchko20a, Kabluchko20b, Kabluchko20c} have shown that this connection with linear statistics of $\beta$-ensembles allows to establish, when the dimension tends to infinity, exact asymptotics for the volume of the Schatten unit balls as well as a weak law of large number for the joint law of the eigenvalues/singular values. 

More precisely,
defining $L_{n,p} \ceq n^{-1}\sum_{i=1}^n \delta_{x_i}$ to be the
random empirical measure under~$\PP_{n,p}$, see \eqref{eq:density}, we recall in
Lemma~\ref{lem:computation-easy} that the computation of $I_q(\UBE)$
can be expressed in terms of the normalization constant~$Z_{n,p}$
appearing
in the definition of~$\PP_{n,p}$, and the moment
$\esp{ \langle
	L_{n,p},h_2\rangle^{\sfrac q2}}$
of the linear statistics associated with the function $h_2(x) \ceq x^2$. 
Thus, the task essentially amounts to estimating asymptotics for the normalizing
constant~$Z_{n,p}$ and for
these linear statistics
associated with the function~$h_2$.
Some of these asymptotics are now well understood in the literature on
random matrices~\cite{BenArous97,Johansson98, Hiai00, Pastur11}.
Our proof of Theorem \ref{thm:limit-isotropic-constant} clarifies the connection between the first order estimates  of $I_q(B(S_p^n)\cap E)$ and the convergence of the empirical distribution towards the equilibrium measure due to Johansson~\cite{Johansson98} that we state in Lemma~\ref{lem:meanconv}, or Hiai-Petz~\cite{Hiai00}, see Lemma~\ref{lem:meanconv2}.
To  prove Theorem~\ref{thm:moment-q}, we need
to understand more precisely the fluctuations of the linear
statistics $\langle L_{n,p},h_2\rangle$ in the self-adjoint case. 
We rely here on
recent results of Bekerman, Lebl\'e, Serfaty~\cite{Bekerman18}
(see also Lambert, Ledoux, Webb~\cite{Lambert19})  concerning the CLT
for fluctuations of $\beta$-ensembles with general potential.
Let us note that in these results, the regularity of the external 
potential~$V$ in the Gibbs measure~$\PP_{n,p}$ plays a prominent role. 
It was assumed to be a polynomial of even degree with positive leading coefficient
in the seminal paper~\cite{Johansson98}.
This condition  has then  been relaxed over the last two decades to include
real-analytic potentials~\cite{Kriecherbauer10,Shcherbina13,Borot13} and, more
recently, potentials of class~$\cC^{r}$ with~$r$ reasonably
large~\cite{Bekerman15,Bekerman18,Lambert19}. However  we are working in a very
specific case where $V(x)$ is proportional to~$|x|^p$. Proposition~\ref{pro:fluct}
states that the result of~\cite{Bekerman18} applies
when $p>3$. It is proved by tracing back in the proof of~\cite{Bekerman18} various places where the regularity of~$V$ is
needed, in particular to prove 
regularity of the solution to a key master equation. 

\subsection{Organization of the paper.}

In Section~\ref{sec:reduction}, we recall the formulas relating the $q$-inertia moment~$I_q$,
the normalizing constant~$Z_{n,p}$, and the expected $q$-moment of the linear statistics $\langle L_{n,p},h_2\rangle$.
We gather, in Section~\ref{sec:moment}, some 
uniform moment bounds which are needed to accommodate the fact that
we study a linear statistics for the unbounded function~$h_2$. This generalizes a classical truncation argument to all possible parameters in the density of~$\PP_{n,p}$.
The proof of our main results is done in Section~\ref{sec:asymptotics}.
We postpone to Section~\ref{sec:properties} the discussion on technical properties of the equilibrium measure and on the regularity of the solution to the master equation.

\section{Reduction to integrals over~$\R^n$}\label{sec:reduction}
A function $F \colon \R^n \to \R$ is said to be symmetric if for every $x \in \R^n$ and every permutation~$\pi$ on $\{1, \ldots, n\}$ we have $F(x_1, \ldots, x_n) = F(x_{\pi(1)}, \ldots, x_{\pi(n)})$. 
Let~$E$ be $\mathcal M_n(\F)$ or
its restriction to self-adjoint matrices.
The following
change of variable formula \cite[Propositions~4.1.1 \&~4.1.3]{Anderson10} makes a connection between the Schatten unit ball $\UBE$ and
the classical unit ball $B_p^n\subset\R^n$: for every symmetric, continuous
function $F\colon\R^n\to\R$,
\begin{equation*}
\int_{\UBE}F\bigl(s(T)\bigr)\,\dd T=c_n\int_{B^n_p}F(\bs x)\,
f_{a,b,c}(\bs x)\dd\bs x,
\end{equation*}
where~$c_n$ is a positive constant,
\[f_{a,b,c}(\bs x)\enspace\ceq\prod_{1\le i<j\le n}\!\!\!\!\left\lvert x_i^a-x_j^a\right\rvert^b
\cdot\:\prod_{i=1}^n{\lvert{x_i\rvert^c}},
\qquad\bs x\ceq(x_1,\ldots,x_n)\in\R^n,\]
defines a positively homogeneous function of degree $d_n-n$, with
\begin{equation}
	d_n \ceq \dim_\R E= ab\,\frac{n(n-1)}2 + (c+1)n
	\label{eq:dimension}
\end{equation}
being the dimension of the subspace~$E$ over~$\R$. The constant~$c_n$
is explicit and related to the volume of the unitary group $U_n(\F)$; it depends only on~$a$, $b$ and~$n$, and its exact and asymptotic values are known and given in Lemma~\ref{lem:cn}:
\begin{equation}
\label{eq:asym-c_n}
\sqrt n \cdot c_n^{1/d_n} \sim\ee^{\frac34} \sqrt{\frac{4\pi}{ab}}
\quad\text{as $n\to\infty$.}
\end{equation}

\begin{table}\small
\begin{tabular}{cccccc}
	\toprule
	\text{Matrix Ensemble}& $\F$ & ~$a$ & $b$ & $c$ 
	\\   
	\midrule
	 real  & $\R$&  2 & 1 & 0
	\\
		 complex & $\C$  & 2 & 2 & 1 
	\\
		 quaternion & $\H$  & 2 & 4 & 3 
	\\   
	\midrule
 real symmetric & $\R$&   1 & 1 & 0  
  \\
  complex Hermitian &$\C$& 1 & 2 & 0 
  \\
  Hermitian quaternionic  &$\H$ & 1 & 4 & 0\\
                        \bottomrule
\end{tabular}
                        \caption{\label{eq:tableau}Possible choices for~$E$, and corresponding parameters.}
\justifying{\noindent\footnotesize In all cases, $b = \beta = \mathrm{dim}_{\R}(\F)$.
	In the self-adjoint cases, $a=1$ and $c=0$.
	In the  cases where $E=\mathcal{M}_n(\F)$,
	the computations depend on the
	singular values, $a=2$ and $c=\beta -1$.
	See
	\cite[Chapter~4]{Anderson10} for details.}
                      \end{table}

\noindent
Combining this change of variable with a classical trick in convexity leads to the following expression.
\begin{lemma}[Change of variables, \cite{Guedon07}\bfseries]
	\label{lem:guedon-paouris}
	For every symmetric, continuous, positively homogeneous function 
	$F\colon\R^n\to\R$
	of degree $k\ge0$, one has
	\[
	\int_{\UBE}F\bigl(s(T)\bigr)\,\dd T
	=\frac{c_n}{\Gamma\left(1+\frac{d_n+k}p\right)}\int_{\R^n}F(\bs x)\ee^{-\|\bs x\|_p^p}
	f_{a,b,c}(\bs x)\,\dd\bs x,
	\]
	where~$\Gamma$ is Euler's Gamma function.
\end{lemma}
We refer to~\cite{Guedon07} for the details of the computations. 
Let
\begin{equation}
\label{eq:gamma_p}
\gamma_p\ceq\frac{\Gamma\left(\frac p2\right)\Gamma\left(\frac12\right)}{2\Gamma\left(\frac{p+1}2\right)}.
\end{equation}
Applying Lemma~\ref{lem:guedon-paouris} with $F\cequiv1$ (homogeneous of degree~$0$)
yields, after a straightforward change of variables,
\begin{equation}
\left\lvert{\UBE}\right\rvert
=
\left(
ab
n\gamma_p\right)
^{\frac{d_n}p}\cdot
\frac{c_n}{\Gamma\left(1+\frac{d_n}p\right)}\cdot Z_{n,p},\label{eq:expr-volume}
\end{equation}
where
\begin{equation}
  Z_{n,p}\ceq\int_{\RR^n}\ee^{-ab n \gamma_p
	\|\bs x\|_p^p}f_{a,b,c}(\bs x)\,\dd\bs x.\label{eq:normalizing-constant}
\end{equation}
Using now Lemma~\ref{lem:guedon-paouris} with
$F(\bs x)\ceq(\sum_{i=1}^n x_i^2)^{q/2}$ (homogeneous of degree~$q$), we obtain,
for~$T$ uniformly distributed in $\UBE$,
\begin{equation}
\esp{\HS T^q}=
\left(
ab
n\gamma_p\right)
^{\frac qp}
\cdot\:
\frac{\Gamma\left(1+\frac{d_n}p\right)}{\Gamma\left(1+\frac{d_n+q}p\right)}
\cdot
n^{\frac q2}
\int_{\RR^n}\left(\frac1n\sum_{i=1}^nx_i^2\right)^{\!\frac q2}
\prb[n,p]{\dd\bs x},
\label{eq:norm-q}
\end{equation}
where~$\prb[n,p]{\dd\bs x}$
is the probability
\begin{equation}
	\prb[n,p]{\dd\bs x}\ceq\frac1{Z_{n,p}}
	\cdot f_{a,b,c}(\bs x)\ee^{-ab
		n\gamma_p
		\|\bs x\|_p^p}\,\dd\bs x,\qquad\bs x\in\RR^n.\label{eq:density}
\end{equation}
We conclude this section by expressing the quantities appearing in Theorems~\ref{thm:limit-isotropic-constant} and~\ref{thm:moment-q} in terms of the integral of $h_2(x)\ceq x^2$ with respect to the random empirical measure
$L_{n,p}\ceq n^{-1}\sum_{i=1}^n\delta_{x_i}$ under~$\PP_{n,p}$.
\begin{lemma}[From inertia to $\beta$-ensembles]
\label{lem:computation-easy}
For every $q>0$, we have
\begin{equation}
  \label{eq:thm1}
\frac{1}{\sqrt {d_n}}\,I_q(\UBE)
= \bigl(1 + o(1)\bigr)
 \frac{\ee^{-\frac1p-\frac34}}{\sqrt{2\pi}}\, Z_{n,p}^{-1/d_n}
  \left(\Esp{ \langle L_{n,p},h_2\rangle^{\frac q2}}\right)^{\!\frac1q}\end{equation}
and 
\begin{equation}
	\label{eq:thm2}
\frac{I_q(\UBE)}{I_2(\UBE)}
=
\left(1 - \frac{q-2}{ab p^2 \, n^2} + o\left(\frac{1}{n^2}\right)\!\right)
\frac{\left(\Esp{\langle L_{n,p},h_2\rangle^{\sfrac q2}}\right)^{\sfrac1q}}{\left(\Esp{\langle L_{n,p},h_2\rangle}\right)^{\sfrac12}}.
\end{equation}
\end{lemma}
\begin{proof}
Observe that
\[
\int_{\RR^n}\left(\frac1n\sum_{i=1}^nx_i^2\right)^{\!\frac q2}
\prb[n,p]{\dd\bs x}
= 
\Esp{ \langle L_{n,p},h_2\rangle^{\!\frac q2}}.
\]
By combining~\eqref{eq:norm-q} and~\eqref{eq:expr-volume},
the quantity appearing in Theorem~\ref{thm:limit-isotropic-constant} is
\begin{align*}
\frac{1}{\sqrt {d_n}}\,I_q(\UBE) & = 
\frac{1}{\sqrt {d_n}} \, \frac{\left(\esp{{\HS T^q}}\right)^{\sfrac1q}}{\lvert\UBE\rvert^{\sfrac{1}{d_n}} }
\\[.4em]
& = 
\frac{1}{\sqrt {d_n}} \, \left(\frac{\Gamma\left(1+\frac{d_n}p\right)}{\Gamma\left(1+\frac{d_n+q}p\right)}\right)^{\!\frac1q}\,
\frac{\Gamma\left(1+\frac{d_n}p\right)^{\!\frac1{d_n}}}{c_n^{1/d_n}\,Z_{n,p}^{1/d_n}}
\sqrt n\,\left(\Esp{ \langle L_{n,p},h_2\rangle^{\frac q2}}\right)^{\!\frac1q},
\\[.4em]
&\sim\sqrt{\frac2{abn}}\ee^{-\frac1p}c_n^{-1/d_n}Z_n^{-1/d_n}\left(\Esp{ \langle L_{n,p},h_2\rangle^{\frac q2}}\right)^{\!\frac1q},
\end{align*}
using that $d_n\sim abn^2/2$ as $n\to\infty$,
	$\Gamma(1+x)^{\frac1x}\sim x/{\ee}$ and $\Gamma(1+x+\alpha)/\Gamma(1+x)\sim x^\alpha$ as $x\to\infty$, for any fixed
	number~$\alpha$.
	Thus, by the asymptotics of~$c_n$ in~\eqref{eq:asym-c_n},
\[\frac{1}{\sqrt {d_
n}}\,I_q(\UBE)
\sim 
 \frac{\ee^{-\frac1p-\frac34}}{\sqrt{2\pi}}\, Z_{n,p}^{-1/d_n}
  \left(\Esp{ \langle L_{n,p},h_2\rangle^{\frac q2}}\right)^{\!\frac1q}.\]

\smallskip
Regarding the quantity involved in Theorem~\ref{thm:moment-q},
two applications of~\eqref{eq:norm-q} give
\begin{align*}
\frac{I_q(\UBE)}{I_2(\UBE)}
&=\frac{(\esp{\HS T^q})^{\sfrac1q}}{(\esp{\HS T^2})^{\sfrac12}}
\notag\\[.4em]
&=
\left(\frac{\Gamma\left(1+\frac{d_n+q}p\right)}{\Gamma\left(1+\frac{d_n}p\right)}\right)^{\!-{\frac1q}}\!\!\!\!\cdot\:\left(\frac{\Gamma\left(1+\frac{d_n+2}p\right)}{\Gamma\left(1+\frac{d_n}p\right)}\right)^{\!\frac12}\!\cdot\:
\frac{\left(\Esp{\langle L_{n,p},h_2\rangle^{\sfrac q2}}\right)^{\sfrac1q}}{\left(\Esp{\langle L_{n,p},h_2\rangle}\right)^{\sfrac12}}.
\end{align*}
We have $d_n \sim ab n^2 /2$ and classical computations on  the~$\Gamma$ function at infinity give that for any fixed $r>0$ and $p \ge 1$, 
\[
\left(\frac{\Gamma\left(1+\frac{d_n+r}p\right)}{\Gamma\left(1+\frac{d_n}p\right)}\right)^{\!-\frac1r}\!\!\!
=
\:\left(\frac{d_n}p\right)^{\!\frac1p}\!\cdot\:\left(1-\frac{p+r}{a b p^2\,n^2}+o\left(\frac1{n^2}\right)\!\right).
\]
Therefore 
\[
\left(\frac{\Gamma\left(1+\frac{d_n+q}p\right)}{\Gamma\left(1+\frac{d_n}p\right)}\right)^{\!-\frac1q}\!\!\!\!\cdot\:\left(\frac{\Gamma\left(1+\frac{d_n+2}p\right)}{\Gamma\left(1+\frac{d_n}p\right)}\right)^{\!\frac12}
= 1 - \frac{q-2}{a b p^2 \,n^2} + o\left(\frac{1}{n^2}\right)
\]
which finishes the proof of~\eqref{eq:thm2}.
\end{proof}
\section{Moment bounds}
\label{sec:moment}
Recall that
$\prb[n,p]{\dd\bs x}$
is the probability whose density is defined on~$\RR^n$ by 
\[
	\prb[n,p]{\dd\bs x}\ceq\frac1{Z_{n,p}}
	\cdot f_{a,b,c}(\bs x)\ee^{-ab
		n\gamma_p
		\|\bs x\|_p^p}\,\dd\bs x,\ \ \hbox{where}\ \
		f_{a,b,c}(\bs x)\:\ceq\!\!\!\prod_{1\le i<j\le n}\!\!\!\!\left\lvert x_i^a-x_j^a\right\rvert^b\cdot\:\prod_{i=1}^n{\lvert{x_i\rvert^c}}.
\]
In this section we consider a more general case where~$a$ is a  positive integer, $b >0$ and $c \geq 0$.
Prior to evaluating their asymptotics, we need to establish some uniform bounds on the quantities
\begin{equation}
	\Bigl(\Esp{\langle L_{n,p},h_r\rangle}\Bigr)^{\!\frac1r}=
	\left(\int_{\RR^n} |x_1|^r
	\PP_{n,p}(\dd\bs x)\right)^{\!\frac1r}
		\label{eq:r-moments}
\end{equation}
where $h_r(x) \ceq \lvert x\rvert^r$.
The evaluation of~\eqref{eq:r-moments} in the case $r=2$ is the crucial part of the proof in~\cite{Koenig98} while it was also studied for larger values of~$r$ in~\cite{Guedon07}.
We prove here a stronger result. Following the method of proof of Theorem~11.1.2 (i) in~\cite{Pastur11}, which corresponds to the case $a=1$ and $c=0$, we establish an upper bound of the tail of the first marginal density of $\prb[n,p]{\dd\bs x}$ denoted by $p_n(x_1)$ and defined by 
\[
p_n(x_1)\ceq\frac1{Z_{n,p}}\int_{\RR^{n-1}}f_{a,b,c}(x_1,\dots,x_n)\ee^{-ab
		n\gamma_p
		\sum_{i=1}^n|x_i|^p}\,\dd x_2\cdots \dd x_n.
\]
\begin{theorem}[Moment and tail bounds on the first marginal]
\label{thm:moment-estimate}
	Let $p\ge1$ and $R>0$. Then there exist constants $C, C', C'', X_1, c_1, c'>0$  depending only on $a,b,c$ and on~$p, R$ such that, for every $n\ge1$:
	\begin{enumerate}[label=(\roman*)]
		\item\label{thm:moment-estimate.i} for every $|x_1|\ge X_1$,
	\[
	p_n(x_1)\le\ee^{-c_1n|x_1|^p}\!;
	\]
	\item\label{thm:moment-estimate.ii} for every $r\ge1$,
	\[
	\left(\int_{\R^n} |x_1|^r
	\,\PP_{n,p}(\dd\bs x)\right)^{\!\frac1r}
	\le X_1+ C\left(\frac{r}{n}\right)^\frac{1}{p},
	\]
	and thus 
	\[
	\forall 1\le r\le Rn, \quad \left(\int_{\R^n} |x_1|^r\PP_{n,p}(\dd\bs x)\right)^{\!\frac1r}\le C'';
	\]
	\item\label{thm:moment-estimate.iii} for every $i= 1, \ldots, n$, and every $B\ge X_1$,
\begin{equation*}
\PP_{n,p} \left( |x_i | \ge B \right) \le C'\ee^{-c'nB^p}\!.
\end{equation*}
\end{enumerate}
\end{theorem}

\begin{proof}
The proof of~(i) goes into three steps.\\
1)\enspace We first prove that there exists a constant $C_1>0$ independent of the dimension such that 
\[
\int_{\RR^n}|x_1|^p\prb[n,p]{\dd\bs x}=\int_{\RR}|x_1|^p\,p_n(x_1)\dd x_1\le C_1.
\] 
The proof of this inequality follows the lines of the proof of Corollary~7(a) in~\cite{Guedon07}. For $t>0$, we define 
\[
g(t)\ceq\int_{\RR^n}f_{a,b,c}(\bs x)\ee^{-tabn\gamma_p\|\bs x\|_p^p}\dd\bs x.
\]
Then $g(1)=Z_{n,p}$. Changing variable by putting $\bs x\gets t^{-1/p}\bs y$ and using that~$f_{a,b,c}$ is positively homogeneous of degree $d_n-n$, where $d_n\ceq abn(n-1)/2+(c+1)n$, we get 
\[
g(t)=t^{-n/p}\int_{\RR^n}f_{a,b,c}(t^{-1/p}\bs y)\ee^{-abn\gamma_p\|\bs y\|_p^p}\dd\bs y=t^{-d_n/p}g(1).
\]
It follows that $g'(1)=-\frac{d_n}{p}g(1)$. On the other hand, differentiating the integral formula which defines~$g(t)$, we also get
\begin{align*}
g'(1)&=-abn\gamma_p\int_{\RR^n}\|\bs x\|_p^p\,f_{a,b,c}(\bs x)\ee^{-abn\gamma_p\|\bs x\|_p^p}\dd\bs x\\[.4em]
&=-abn^2\gamma_p\int_{\RR^n}|x_1|^p\, f_{a,b,c}(\bs x)\ee^{-abn\gamma_p\|\bs x\|_p^p}\dd\bs x.
\end{align*}
We conclude that
\[
\int_{\RR^n}|x_1|^p\prb[n,p]{\dd\bs x}=\int_{\RR}|x_1|^p\,p_n(x_1)\dd x_1=-\frac{g'(1)}{g(1)abn^2\gamma_p}=\frac{d_n}{abpn^2\gamma_p}.
\]
It follows that $\int_{\RR^n}|x_1|^p\prb[n,p]{\dd\bs x}\to 1/(2p\gamma_p)$ as $n\to\infty$ and therefore is upper bounded by a constant~$C_1$.

\noindent 2)\enspace In the second step we prove that there exist $c_2,X_2>0$ independent of the dimension such that, for all $|x_1|\ge X_2$,
\[
p_n(x_1)\le\frac{Z_{n-1,p}}{Z_{n,p}}\ee^{-c_2n|x_1|^p}\!.
\]
For $\bs x\in\RR^n$, we denote 
\[
g_n(\bs x)\ceq Z_{n,p}\,\frac{\prb[n,p]{\dd\bs x}}{\dd\bs x}\:=\!\!\!\prod_{1\le i<j\le n}\!\!\!\!\left\lvert x_i^a-x_j^a\right\rvert^b\cdot\:\prod_{i=1}^n{\lvert{x_i\rvert^c}}\cdot\ee^{-tabn\gamma_p\|\bs x\|_p^p}.
\]
For $\tilde{\bs x}\ceq(x_2,\dots,x_n)\in\RR^{n-1}$ and $\bs x\ceq(x_1, \tilde{\bs x})$, using that $n\|\bs x\|_p^p=n|x_1|^p+(n-1)\|\tilde{\bs x}\|_p^p+\|\tilde{\bs x}\|_p^p$ we get
\[
g_n(\bs x)=g_n(x_1,\tilde{\bs x})=g_{n-1}(\tilde{\bs x})\ee^{-abn\gamma_p|x_1|^p}\ee^{-ab\gamma_p\|\tilde{\bs x}\|_p^p}|x_1|^c\prod_{i=2}^n\left\lvert x_i^a-x_1^a\right\rvert^b.
\]
Hence 
\begin{equation}
	\label{eq:marginals}
	\begin{aligned}
		&
\frac{1}{Z_{n-1,p}}\int_{\RR^{n-1}}g_n(x_1,\tilde{\bs x})\dd\tilde{\bs x}\\
&\qquad
=|x_1|^c\ee^{-abn\gamma_p|x_1|^p}\int_{\RR^{n-1}}\ee^{-ab\gamma_p\|\tilde{\bs x}\|_p^p}\prod_{i=2}^n\left\lvert x_i^a-x_1^a\right\rvert^b\prb[n-1,p]{\dd\tilde{\bs x}}.
\end{aligned}
\end{equation}
Using that for any $x,y\in\RR$ one has $|x-y|\le|x|+|y|\le (1+|x|)(1+|y|)$, and letting $D_2\ceq\max_{t\in\RR} \ee^{-ab\gamma_pt^p}(1+|t|^{a})^b$, we have 
\[
\frac{Z_{n,p}}{Z_{n-1,p}}\,p_n(x_1)=\frac{1}{Z_{n-1,p}}\int_{\RR^{n-1}}g_n(x_1,\tilde{\bs x})\dd\tilde{\bs x}\le  |x_1|^c\ee^{-abn\gamma_p|x_1|^p}(1+|x_1|^{a})^{(n-1)b}D_2^{n-1}.
\]
We conclude that for $c_2=ab\gamma_p/2$ there exists $X_2>0$ independent of the dimension such that, for $|x_1|\ge X_2$,
\[
p_n(x_1)\le\frac{Z_{n-1,p}}{Z_{n,p}}\ee^{-c_2n|x_1|^p}\!.
\]

\noindent3)\enspace In the third step, we establish that there exists $c_3>0$ such that $\frac{Z_{n-1,p}}{Z_{n,p}}\le \ee^{-c_3n}\!$.
Integrating equation (\ref{eq:marginals}) with respect to $x_1\in\RR$ we get 
\begin{equation}
\label{eq:intmarginals}
\begin{aligned}
&\frac{1}{Z_{n-1,p}}\int_{\RR^{n}}g_n({\bs x})\dd{\bs x}\\
&\qquad=\int_{\RR}|x_1|^c\ee^{-abn\gamma_p|x_1|^p}\int_{\RR^{n-1}}\ee^{-ab\gamma_p\|\tilde{\bs x}\|_p^p}\prod_{i=2}^n\left\lvert x_i^a-x_1^a\right\rvert^b\prb[n-1,p]{\dd\tilde{\bs x}}\dd x_1.
\end{aligned}
\end{equation}
Applying Jensen's inequality to the integral on~$\RR^{n-1}$ with the probability measure $\PP_{n-1,p}$ and the convex function being the exponential, and applying also the bound obtained in Step~1, we have
\begin{flalign*}
&\int\ee^{-ab\gamma_p\|\tilde{\bs x}\|_p^p}\prod_{i=2}^n\left\lvert x_i^a-x_1^a\right\rvert^b\prb[n-1,p]{\dd\tilde{\bs x}}\\[.4em]
&\hspace{6em}\ge\exp\int\left(-ab\gamma_p\|\tilde{\bs x}\|_p^p+b\sum_{i=2}^n\ln{\left\lvert x_i^a-x_1^a\right\rvert}\right)\prb[n-1,p]{\dd\tilde{\bs x}}\\[.4em]
&\hspace{6em}\ge \exp\left(-(n-1)ab\gamma_pC_1+b\int\sum_{i=2}^n\ln{\left\lvert x_i^a-x_1^a\right\rvert}\prb[n-1,p]{\dd\tilde{\bs x}}\right).
\end{flalign*}
Plugging this inequality into Equation~\eqref{eq:intmarginals} and noticing that the left-hand side is equal to $Z_{n,p}/Z_{n-1,p}$ we get 
\[
\frac{Z_{n,p}}{Z_{n-1,p}}\ge \ee^{-(n-1)ab\gamma_pC_1}\int_{-\frac{1}{2}}^{\frac{1}{2}}|x_1|^c\ee^{-abn\gamma_p|x_1|^p}\exp\left(b\int\sum_{i=2}^n\ln\left\lvert x_i^a-x_1^a\right\rvert\prb[n-1,p]{\dd\tilde{\bs x}}\right)\dd x_1.
\]
Applying again Jensen's inequality but with the uniform probability measure on $\cc{-\frac{1}{2}}{\frac{1}{2}}$ we deduce that 
\begin{flalign*}
	&&\frac{Z_{n,p}}{Z_{n-1,p}}\ge\ee^{-(n-1)ab\gamma_pC_1}\exp
\int_{-\frac{1}{2}}^{\frac{1}{2}}\Biggl(c\ln|x_1|-abn\gamma_p|x_1|^p\hspace{13em}\\
&&+b\int\sum_{i=2}^n\ln\left\lvert x_i^a-x_1^a\right\rvert\prb[n-1,p]{\dd\tilde{\bs x}}\Biggr)\dd x_1
.\qquad
\end{flalign*}
We then use Fubini's theorem and thus want to estimate from below the
function~$g_a$ defined for $a\ge1$ being an integer and $x\in\RR$ by
\[
g_a(x)\ceq\int_{-\frac{1}{2}}^{\frac{1}{2}}\ln{\lvert x^a-y^a\rvert}\,\dd y.
\]
Namely we shall prove that $g_a(x)\ge -a(2\ln2+1)$. First notice that for~$a$ even the function $(x,y)\mapsto x^a-y^a$ is even in both variables, and for~$a$ odd, splitting the integral and changing variable one has 
\[
g_a(x)=\int_0^{\frac{1}{2}}\bigl(\ln{\lvert x^a-y^a\rvert}+\ln{\lvert x^a+y^a\rvert}\bigr)\,\dd y=\int_0^{\frac{1}{2}}\ln{\lvert x^{2a}-y^{2a}\rvert}\,\dd y=\frac{1}{2}\,g_{2a}(x).
\]
We are thus reduced to proving the lower bound of~$g_a$ in the case where~$a$ is even. Then the function~$g_a$ is even so we may assume that $x\ge0$ and we have 
\[
g_a(x)=2\int_0^{\frac{1}{2}}\ln{\lvert x^a-y^a\rvert}\,\dd y.
\]
For~$x\ge \frac{1}{2}$ the function~$g_a$ is increasing hence $g_a(x)\ge g_a(1/2)$. Moreover for every $x,y>0$ one has $\vert x^{a}-y^{a}\rvert\ge y^{a-1}\lvert x-y\rvert$, thus 
\[
g_a(x)\ge2\int_0^{\frac{1}{2}}\bigl((a-1)\ln y+\ln{\lvert x-y\rvert}\bigr)\,\dd y.
\]
Let~$\varphi$ be the convex function defined by $\varphi(x)=x\ln x$ for $x\ge0$. A simple calculation shows that, for every $0\le x\le\frac{1}{2}$,
\[
g_a(x)\ge -(a-1)(\ln 2+1)+2\varphi(x)+2\varphi\left(\frac{1}{2}-x\right)-1.
\]
From the convexity of~$\varphi$, one has 
\[
\varphi(x)+\varphi\left(\frac{1}{2}-x\right)\ge 2\varphi\left(\frac{x+\frac{1}{2}-x}{2}\right)=2\varphi\left(\frac{1}{4}\right)=-{\ln2}.
\]
We conclude that $g_a(x)\ge -a(2\ln2+1)$. Applying this inequality in the lower bound of $Z_{n,p}/Z_{n-1,p}$ we deduce that
\begin{align*}
\frac{Z_{n,p}}{Z_{n-1,p}}&\ge\exp\left(-(n-1)ab\gamma_pC_1-(2\ln 2+1)\bigl(c+(n-1)ab)\bigr)-\frac{nab\gamma_p}{(p+1)2^p}\right)\\[.4em]&\ge \ee^{-c_3n},
\end{align*}
where $c_3\ceq ab\gamma_pC_1+(2\ln 2+1)(c+ab))+\frac{ab\gamma_p}{(p+1)2^p}$, which finishes our third step. The final conclusion follows from combining the steps~2 and~3 which give that for every $|x_1|\ge X_2$, one has
\[
p_n(x_1)\le\frac{Z_{n-1,p}}{Z_{n,p}}\ee^{-c_2n|x_1|^p}\le \ee^{-c_3n-c_2n|x_1|^p}\!.
\]
Hence there exist $c_1,X_1>0$ such that $p_n(x_1) \le \ee^{-c_1n|x_1|^p}$
for all $|x_1|\ge X_1$.

\noindent(ii)\enspace The upper bound for the moment using the bound for the tail is standard and runs as follows. We first cut the integral into two parts and use the bound proved in (i):
\begin{align*}
\int_{\R^n} |x_1|^r\PP_{n,p}(\dd\bs x)=\int_\R |t|^r\,p_n(t)\,\dd t
&\le X_1^r \int_{|t|\le X_1}p_n(t)\dd t+\int_{|t|\ge X_1}|t|^r\ee^{-c_1n|t|^p}\dd t\\
&\le X_1^r +2\int_0^{\infty}t^r\ee^{-c_1nt^p}\dd t.
\end{align*}
Changing variable, this last integral may be written as follows:
\[
\int_0^{\infty}t^r\ee^{-c_1nt^p}\dd t=\frac{1}{p(c_1n)^\frac{r+1}{p}}\int_0^{\infty}s^{\frac{r+1}{p}-1}\ee^{-s}\dd s=\frac{1}{p(c_1n)^\frac{r+1}{p}}\,\Gamma\left(\frac{r+1}{p}\right).
\]
Using a standard bound on the Gamma function we conclude that there exists $C>0$ such that 
\[
	\left(\int_{\R^n} |x_1|^r
	\,\PP_{n,p}(\dd\bs x)\right)^{\!\frac1r}
	\le X_1+ \left(\frac{2}{p(c_1n)^\frac{r+1}{p}}\,\Gamma\left(\frac{r+1}{p}\right)\!\right)^{\!\frac{1}{r}}\le X_1+ C\left(\frac{r}{n}\right)^{\!\frac{1}{p}}.
	\]

\noindent(iii)\enspace Using the bound~(i), for $B>X_1$ and choosing $c'\ceq c_1/2$, we have
\[ 
\PP_{n,p} \left( \lvert x_i\rvert \ge B \right)\le2 \int_B^{\infty}\ee^{-c_1nt^p}\dd t\le 2 \ee^{-c'nB^p}\int_0^{\infty}\ee^{-c'nt^p}\dd t\le C'\ee^{-c'nB^p}\!.\qedhere
\]
 
\end{proof}

Because of the above uniform bounds, the random measure~$L_{n,p}$
is mostly concentrated on a compact interval. As a result,
testing~$L_{n,p}$ against a general function~$f$ is on average not very different
from using a truncated version of~$f$ instead.
\begin{corollary}[Truncation argument]
\label{cor:troncature}
	Choose $X_1\ge 1$ from Theorem~\ref{thm:moment-estimate}. Let $B\ge X_1$ and let
	$\phi\colon\RR\to\cc{0}{1}$ be a smooth, compactly supported function such that~%
$\phi\equiv1$ on $\cc{-B}{B}$.
	Then
	for all continuous, polynomially bounded functions $f,g\colon\RR\to\RR$,
	one has
	\[\Esp{\Bigl\lvert
	g\bigl(\langle L_{n,p},f\rangle\bigr)-
    g\bigl(\langle L_{n,p},f\phi\rangle\bigr)\Bigr\rvert}
    =O(\alpha^n),\]
    where $\alpha\ceq\ee^{-\frac{c'B^p}{2}}<1$.
\end{corollary}
\begin{proof}
As~$f$ and~$g$ are polynomially bounded, there exist constants
$A_1, A_2, s, t \ge 1$ such that 
for every $y \in \RR$, $|f(y)| \le A_1 + |y|^s$ and $|g(y)| \le A_2 + |y|^t$. Hence, we can find constants 
$A,r\ge1$ such that, for every $\bs x\in\RR^n$,
\[
\left|g\left(\langle L_{n,p}(\bs x),f\rangle\right)\right|
= 
\left\lvert g  \left(\frac{1}{n} \sum_{i=1}^n f(x_i) \right) \right\rvert
\le
A +  \frac{1}{n} \sum_{i=1}^n |x_i|^r
= 
A + \langle L_{n,p}(\bs x),h_r\rangle
\]
(e.g., $A\ceq2^tA_1+A_2$ and $r\ceq st$, using Jensen's inequality).
Since  $0 \le \phi \le 1$ on~$\RR$, we get the same bound for 
\[
\lvert g(\langle L_{n,p}(\bs x),f\phi\rangle)\rvert
\le 
A+\langle L_{n,p}(\bs x),h_r\rangle.
\]
As $\phi \equiv 1$ on $\cc{-B}{B}$,
\[
\left\lvert g\left(\langle L_{n,p}(\bs x),f\rangle\right)-
g\left(\langle L_{n,p}(\bs x),f\phi\rangle\right)\right\rvert
\le2\ind{\{\exists i:\lvert x_i\rvert>B\}}
\bigl(A+\langle L_{n,p}(\bs x),h_r\rangle\bigr).
\]
Since~$r$ is a constant, we have  by  Theorem~\subref{thm:moment-estimate}{ii}, 
\[
\Esp{\langle L_{n,p}(\bs x),h_r\rangle^2}
= 
\Esp {\left( \frac{1}{n} \sum_{i=1}^n |x_i|^r \right)^{\!2} }\!
\le\: 
\Esp {\frac{1}{n} \sum_{i=1}^n |x_i|^{2r} }
\le
C''^{2r}.
\]
Combining Theorem~\subref{thm:moment-estimate}{iii} with a union bound,
\[
\PP_{n,p} \left( \exists i:\lvert x_i\rvert>B\right)
\le
nC' \ee^{-c'nB^p}.
\]
Therefore, by Cauchy-Schwarz inequality, we conclude that
\[\Esp{\Bigl\lvert
	g\bigl(\langle L_{n,p},f\rangle\bigr)-
	g\bigl(\langle L_{n,p},f\phi\rangle\bigr)\Bigr\rvert}
\le
2nC'\ee^{-c'nB^p}(A+C''^r).
\]
This is $O(\alpha^n)$ with e.g.\ $\alpha=\ee^{-\frac{c'B^p}{2}}$.
\end{proof}

\section{Asymptotics}\label{sec:asymptotics}
In this section, we prove Theorems~\ref{thm:limit-isotropic-constant}
and~\ref{thm:moment-q} by completing the asymptotic expansions initiated
in Lemma~\ref{lem:computation-easy}. As previously explained, we will appeal to the literature on random matrices,
as~$L_{n,p}$ corresponds to the empirical distribution of the so called~$\beta$-ensemble with potential
$V(x)\propto\lvert x\rvert^p$ and~$Z_{n,p}$ is the so called partition
function. We split the discussion into two cases,
according to  Table~\ref{eq:tableau}. In the first case ($a=1$),
we prove both Theorems~\ref{thm:limit-isotropic-constant} and~\ref{thm:moment-q},
where matrices in $\UBE$ all have symmetries with respect to
their diagonal and~$L_{n,p}$ thus corresponds to empirical distributions of real
eigenvalues. The second case ($a=2$) pertains to
Theorem~\ref{thm:limit-isotropic-constant} only and is treated more
conveniently by working with $\RR+$-valued measures.

\subsection{The self-adjoint case ($a=1$)}\label{sec:asymptotics.1}
In view of Table~\ref{eq:tableau}, we have $a=1$, $b = \mathrm{dim}_{\R}(\F)$ and $c=0$.  Hence, from~\eqref{eq:density}, the probability~$\PP_{n,p}$  can be written
\begin{equation}
	\PP_{n,p}(\dd\bs x)=\frac1{Z_{n,p}}\cdot\ee^{-\frac b2n^2H_{n,p}(\bs x)}\,\dd\bs x,
	\label{eq:hamiltonian}
\end{equation}
with the Hamiltonian
\[
	H_{n,p}(\bs x) \ceq\frac2
	n\sum_{i=1}^n\gamma_p
	\lvert x_i\rvert^p
	-\frac1{n^2}\sum_{i\neq j}
	\log{\left\lvert x_i-x_j\right\rvert},
\]
where~$\gamma_p$ is defined in~\eqref{eq:gamma_p}.
Heuristically this Hamiltonian is approximated as
		$n\to\infty$ by
\[		
	\cI_p(\mu) \ceq2
	\int_{\RR}\gamma_p
	\lvert x\rvert^p\,\mu(\dd x)-\iint_{\RR^2_{\neq}}\log{\left\lvert x-y\right\rvert}\,\mu(\dd x)\mu(\dd y),
\]
where $\RR^2_{\neq}\ceq\{(x,y)\in\RR^2\colon x\neq y\}$,
provided that the probability measure~$\mu$ is sufficiently close to the
empirical measure $L_{n,p}(\bs x)\ceq n^{-1}\sum_{i=1}^n\delta_{x_i}$. Thus,
because of~\eqref{eq:hamiltonian}, we expect that~$L_{n,p}$
concentrates
as $n\to\infty$ around a probability measure minimizing the functional~$\cI_p$.
This idea constitutes the cornerstone of large deviation principles for the
empirical spectral distribution of large random matrices, as originally
derived by Ben Arous and Guionnet~\cite{BenArous97}. The next two lemmas
formalize what we need. We refer to the textbooks \cite[Chapter~6]{Deift99} and
\cite[Section~2.6]{Anderson10} or
the recent article~\cite{Dupuis20} for deeper results.
\begin{lemma}[Equilibrium measure]
\label{lem:minimizer}
	There exists a unique element~$\mu_p$ minimizing the
	functional~$\cI_p$ over the space $\cP(\RR)$ of real probability measures:
	\[\inf_{\mu\in\cP(\RR)}\cI_p(\mu)=\cI_p(\mu_p).\]
	Furthermore, $\mu_p$ has the following properties:
	\begin{enumerate}[label=(\roman*)]
		\item\label{lem:minimizer.i}~$\mu_p$ is compactly supported, with
		support $\cc{-1}{1}$;
		\item\label{lem:minimizer.ii}~$\mu_p$ is absolutely continuous with respect to the Lebesgue measure,
		with density
		\[f_p(x)\ceq\frac{p\lvert x\rvert^{p-1}}{\pi}\:
		\int_{\lvert x\rvert}^1
		\frac{v^{-p}}{\sqrt{1-v^2}}\,\dd v,\qquad \lvert x\rvert\le1;\]
		\item\label{lem:minimizer.iii} $\cI_p(\mu_p)=\log2+\frac{3}{2p}$.
	\end{enumerate}
\end{lemma}
\begin{proof}
	The existence and uniqueness of the minimizer~$\mu_p$
	follows from \cite[Theorem~I.1.3]{Saff97} with the weight function
	$w\colon x\mapsto\exp(-\gamma_p\lvert x\rvert^p)$ on~$\RR$. Properties (i)--(iii)
	are provided by \cite[Theorem~IV.5.1]{Saff97}. We have rewritten the density
	with the change of variable $v\gets\lvert x\rvert/u$.
\end{proof}
\begin{lemma}[Convergence towards the equilibrium measure]
\label{lem:meanconv}
Let~$L_{n,p}$ be the empirical measure of~$\bs x$ drawn from the probability~$\PP_{n,p}$ defined in~\eqref{eq:hamiltonian}. Then:
	\begin{enumerate}[label=(\roman*)]
		\item\label{lem:meanconv.i} For every continuous, polynomially bounded functions $f,g\colon\RR\to\RR$,
	it holds that
	\[\lim_{n\to\infty}\Esp g\left(\langle L_{n,p},f\rangle\right)=g(\langle\mu_p,f\rangle).\]
    \item\label{lem:meanconv.ii} We have
	\[{\lim_{n\to\infty}{-\frac1{d_n}}\log{Z_{n,p}}}=\log2+\frac{3}{2p},\]
	where the normalizing constant~$Z_{n,p}$ is defined in~\eqref{eq:hamiltonian}.
	\end{enumerate}
\end{lemma}
\begin{proof}
	(i)\enspace
By
	Corollary~\ref{cor:troncature},
	we may reduce to~$f$ bounded. Thus the random variable
	$g(\langle L_{n,p},f\rangle)$ is uniformly bounded, and letting
	$\Delta_{n,p}(f)\ceq\langle L_{n,p},f\rangle-\langle\mu_p,f\rangle$ we
	see from the Borel-Cantelli lemma and the dominated convergence theorem
	that it suffices to show that
	\[\sum_{n=1}^\infty\prb{\lvert\Delta_{n,p}(f)\rvert>\varepsilon_n}<\infty,\]
	for some null sequence~$(\eps_n)$. Thanks to \cite[Theorem~2.1]{Johansson98}, we have
	that
	\[\alpha_n\ceq\frac1n\log\esp{\exp\bigl(n\lvert \Delta_{n,p}(f)\rvert\bigr)}
	\xrightarrow[n\to\infty]{}0.\]
	Choosing now~$\eps_n$ as, e.g., $\eps_n\ceq\alpha_n/2+(\log n)/n$, we see from
	Markov's inequality that
	\[\PP\left(\lvert\Delta_{n,p}(f)\rvert>2\eps_n\right)=
	\PP\left(\exp(n\lvert\Delta_{n,p}(f)\rvert)>\exp(2n\eps_n)\right)
	=O\left(e^{-(2\eps_n-\alpha_n)n}\right)=O\left(\frac1{n^2}\right),\]
	and the first point of the lemma is proved.

    \noindent
	(ii)\enspace This point follows from~\cite[Corollary~4.3]{Johansson98}
	and Lemma~\subref{lem:minimizer}{iii}.
\end{proof}

We now have all the ingredients to give a simple proof of Theorem~\ref{thm:limit-isotropic-constant}
in the case $a=1$.
\begin{proof}[Proof of Theorem~\ref*{thm:limit-isotropic-constant}, case $a=1$]
We start by recalling~\eqref{eq:thm1}: as $n\to\infty$,
\[
	\frac{1}{\sqrt {d_n}}\,I_q(\UBE)
	\sim
	\frac{\ee^{-\frac1p-\frac34}}{\sqrt{2\pi}}\, Z_{n,p}^{-1/d_n} \,
	\left(\Esp{ \langle L_{n,p},h_2\rangle^{\frac q2}}\right)^{\!\frac1q}.
\]
By Lemma~\ref{lem:meanconv}, we have
$Z_{n,p}^{-1/d_n}\to2\ee^{\frac3{2p}}$ and
 \[
 \lim_{n\to\infty}\Esp{\langle L_{n,p},h_2\phi\rangle^{\frac q2}}
 =
\langle\mu_p,h_2\phi\rangle^{\frac q2}.
\]
Plugging in the value of $\langle\mu_p,h_2\rangle$
computed in Lemma~\ref{lem:moment-h2},
we conclude that
\[
\lim_{n\to\infty}\frac{1}{\sqrt {d_n}}\,I_q(\UBE)
=\ee^{\frac{1}{2p}-\frac34}\sqrt{\frac{p}{\pi(p+2)}}.\qedhere
\]
\end{proof}

The proof of Lemma~\ref{lem:meanconv} informs us rather poorly on the
speed of convergence of the linear statistics $\langle L_{n,p},f\rangle$
towards $\langle\mu_p,f\rangle$.
In fact, a necessary condition for Theorem~\ref{thm:moment-q} to
hold is
\[\Var{\langle L_{n,p},h_2\rangle}=O\left(\frac1{n^2}\right),\]
so we expect
that $\langle L_{n,p},h_2\rangle-\langle\mu_p,h_2\rangle=O(n^{-1})$.
Besides, to identify the actual limit in Theorem~\ref{thm:moment-q},
we need to understand precisely the asymptotic behavior of the fluctuations
$n(\langle L_{n,p},h_2\rangle-\langle\mu_p,h_2\rangle)$.

There has been an increasing wealth of literature on such fluctuations
of linear statistics. Given a general potential $V$, let
$L_n\ceq n^{-1}\sum_{i=1}^n\delta_{x_i}$ be the empirical distribution associated with an ensemble $\bs x\ceq(x_1,\ldots,x_n)\in\RR^n$ of particles which are
subject to the confining potential~$V$ and pairwise repulsive logarithmic
interaction. Then, the convergence to a Gaussian distribution
as $n\to\infty$ of the random variable
\begin{equation}
	F_n(\xi)\ceq n\Bigl(\langle L_n,\xi\rangle - \langle\mu_V,\xi\rangle\Bigr),
	\label{eq:fluct}
\end{equation}
where~$\xi$ is a given test function and $\mu_V$ is the equilibrium distribution, has been widely studied. The regularity of the external potential $V$ plays a prominent role. It was assumed to be a polynomial of even degree with positive leading coefficient in the seminal paper of Johansson~\cite{Johansson98}. This condition has then been relaxed during the last two decades to include
real-analytic potentials~\cite{Kriecherbauer10,Shcherbina13,Borot13} and, more
recently, potentials of class~$\cC^{r}$ with~$r$ reasonably
large~\cite{Bekerman15,Bekerman18,Lambert19}.

In our setting (specifically, of Theorem~\ref{thm:moment-q}), the potential
is $V=V_p\ceq2\gamma_ph_p$ with a lack of regularity at $0$, the $\beta$-ensemble~$\bs x$ has the
distribution~$\PP_{n,p}$ with $a=1$, $b=\beta$, and $c=0$
(cf.~\eqref{eq:hamiltonian}), and the equilibrium measure~$\mu_V$ is of course
the probability distribution~$\mu_p$ of Lemmas~\ref{lem:minimizer}
and~\ref{lem:meanconv}. We only need to establish the central limit theorem for $\xi=h_2$,
that is the convergence of $F_n(h_2)$. As~$V_p$ is not always a polynomial nor
real-analytic, we choose to work with the currently most general version due to
Bekerman, Leblé and Serfaty~\cite{Bekerman18}. Like in Johansson~\cite{Johansson98},
the central limit theorem is obtained by establishing convergence of the moment
generating function (a.k.a.\ Laplace transform) of~\eqref{eq:fluct}.

\begin{proposition}[Fluctuations of the linear statistics]
\label{pro:fluct}
	Let $p\in\oo{3}{\infty}$ and let \[F_{n,p}\ceq n\Bigl(\langle L_{n,p},h_2\phi\rangle-\langle\mu_p,h_2\rangle\Bigr),\]
	with~$\phi$ being the truncation function of Corollary~\ref{cor:troncature}.
	Then $(F_{n,p}^2)^{\phantom{2}}_{n\ge1}$ is uniformly integrable, i.e.,
	\[\underset{K\to\infty}{\operatorname{lim\vphantom{p}}}
	\limsup_{\vphantom{K}n\to\infty}\:\Esp{F_{n,p}^2\ind{\{F_{n,p}^2>K\}}}=0.\]
	Furthermore, $\lim\limits_{n\to\infty}\Var F_{n,p}=\frac1{4b}$.
\end{proposition}

\begin{remark}
When $p \ge 6$, this proposition is a direct consequence of \cite[Theorem~1]{Bekerman18}.
Indeed, we are in the so called ``one cut'' regime which corresponds to 
$\mathsf m=\mathsf n=\mathsf k=0$, and we remark that $\xi = h_2 \phi$ is $\cC^\infty$ and $V = V_p$ is $\cC^6$ when $p \ge 6$.

Observe that for $p \ge 8$, this is also a direct consequence of the CLT
 of Lambert, Ledoux and Webb \cite[Theorem~1.2]{Lambert19}, which they derived
alternatively using Stein's method.
\end{remark}
\begin{proof}
	To explain why it is enough to assume $p>3$, and for the sake of clarity, we reproduce the key arguments of Bekerman, Leblé and Serfaty~\cite{Bekerman18}, introducing in particular
	 the key master equation~\eqref{eq:sol-master}, and underlining where the regularity of its solution is required. 
	
	The goal is to prove that, when $p\in\oo{3}{\infty}$, the moment generating function of~$F_{n,p}$
	converges to that of a Gaussian variable~$N$ with variance~$1/4b$, that is,
	there exists $m_p\ceq m_p(h_2\phi)\in\RR$ such that
	\begin{equation}
		\lim_{n\to\infty}\Esp{\ee^{sF_{n,p}}}
		= \exp\left(sm_p+\frac{s^2}{2}\left(\frac1{4b}\right)\!\right)
		\label{eq:cv-laplace}
	\end{equation}
	holds for all $s\in\RR$ with~$\lvert s\rvert$ sufficiently small. 
	It entails the
	convergence in distribution of~$F_{n,p}$ towards $N\sim\cN(m_p,\frac1{4b})$,
	together with the convergence of all moments
	$\Esp F_{n,p}^k\to\EE N^k$, $k\in\NN$.
	In particular, $(F_{n,p}^2)^{\phantom{2}}_{n\ge1}$ will be uniformly
	integrable and $\lim_{n \to +\infty} \Var F_{n,p} = 1/4b.$.

First, by
Theorem~\subref{thm:moment-estimate}{iii}
and union bound, we can choose $B>1$ sufficiently large such that
$\PP_{n,p}(\exists i\le n:\lvert x_i\rvert\ge B)\le C'n\ee^{-c'nB^p}$ for
some constants $c',C'>0$.
Observing that $\|F_{n,p}(h_2\phi)\|_\infty=O(n)$ we deduce that
as soon as~$\lvert s\rvert$ is sufficiently small,
\[\lim_{n\to\infty}\Enp\ee^{sF_{n,p}}\!\ind{\{\exists i\le n:\lvert 
	x_i\rvert\ge B\}}=0.\]
Thus, we may restrict to $\bs x\in U_0^n$ with $U_0\ceq\oo{-B}{B}$.
Now, the strategy, adopted in~\cite{Bekerman18} and developed already
in~\cite{Johansson98}, is to perturb the Hamiltonian~$H_{n,p}$ in~\eqref{eq:hamiltonian} as follows:
\begin{equation*}
	H_{n,p}^t(\bs x)\ceq H_{n,p}(\bs x)+\frac tn\sum_{i=1}^nx_i^2
	=\frac1n\sum_{i=1}^n\left(V_p(x_i)+tx_i^2\right)-\frac1{n^2}\sum_{1\le i\neq j\le n}\!\!\!\!\log{\lvert x_i-x_j\rvert},
\end{equation*}
where we have set $t\ceq-\frac{2s}{bn}$. Then, we write
\begin{equation*}
	\Enp\ee^{sF_{n,p}}
	=o(1)+\ee^{-sn\langle\mu_p,h_2\rangle}\Enp\ee^{sn\langle L_{n,p},h_2\rangle}\!\ind{U_0^n}
	=o(1)+\frac{\ee^{-sn\langle\mu_p,h_2\rangle}}{Z_{n,p}}\int_{U_0^n}\ee^{-\frac b2n^2H_{n,p}^t(\bs x)}\dd\bs x.
\end{equation*}
Next, one applies in the last integral a change of variables
$x_i\gets\vartheta_t(y_i)$ for $1\le i\le n$ and
for some well-chosen~$\cC^1$-diffeomorphism $\vartheta_t\colon U_t\to U_0$
depending on~$t$. We get
\begin{align*}
	\Enp\ee^{sF_{n,p}}
	&=o(1)+\frac{\ee^{-sn\langle\mu_p,h_2\rangle}}{Z_{n,p}}
	\int_{U_t^n}\ee^{-\frac b2n^2H_{n,p}^t\circ\vartheta_t(\bs y)}\prod_{i=1}^n\lvert \vartheta_t'(y_i)\rvert\,\dd\bs y\\[.4em]
	&=o(1)+\ee^{-sn\langle\mu_p,h_2\rangle}\Enp\ee^{-\frac b2n^2(H_{n,p}^t\circ\vartheta_t-H_{n,p})+n\langle L_{n,p},\log{\lvert\vartheta_t'\rvert}\rangle}\!\ind{U_t^n}.
\end{align*}
The idea is that, for a judicious choice of~$\vartheta_t$, the exponent
in the last expectation becomes simple enough to establish the convergence
towards the moment generating function of a Gaussian distribution.
The diffeomorphism is chosen as $\vartheta_t(y)\ceq y+t\psi_p(y)$
where, in our setting, $\psi_p\colon\RR\to\RR$
is the solution to the equation
\begin{equation}
	\Xi_p\psi=\frac12h_2+c_{h_2}\label{eq:sol-master}
\end{equation}
for some constant~$c_{h_2}$, where~$\Xi_p$ is the so called ``master operator''
acting on~$\cC^1$ functions and defined for $x\in\R$ by
\[\Xi_p\psi(x)=-\frac12
V'_p(x)\psi(x)+\int\frac{\psi(x)-\psi(y)}{x-y}\,\mu_p(\dd y).\]
We show in Section~\ref{sec:proof-fluct} that~$\psi_p$ is an odd function, we
give its explicit expression and check that it is of
class~$\cC^{\lceil p\rceil -1}\subset\cC^1$.
Then, by the local
inversion theorem, $\vartheta_t=\operatorname{Id}+t\psi_p$ becomes a
$\cC^1$-diffeomorphism
$U_t\to U_0$ provided $|t|=2|s|/bn\le\tau$ is sufficiently small (equivalently, $n$
is sufficiently large), and we have
\begin{equation}
	\Enp\ee^{sF_{n,p}}
	=o(1)+\ee^{-sn\langle\mu_p,h_2\rangle}\Esp\ee^{-\frac b2\,n^2\left(H_{n,p}^t\circ\vartheta_t-H_{n,p}\right)
		+n\langle L_{n,p},\log{(1+t\psi_p')}\rangle}\ind{U_t^n},\label{eq:expr-fluct}
\end{equation}
where $U_t\ceq\oo{-A_t}{A_t}$ and $\vartheta_t(A_t)=A_t+t\psi_p(A_t)=B$. Furthermore, we know from the implicit function theorem
that $A_t,\,\lvert t\rvert\le\tau,$ depends
continuously on~$t$, and thus choosing~$\tau$ small enough we can assume that
$1<\inf_{\lvert t\rvert\le\tau}A_t<\sup_{\lvert t\rvert\le \tau}A_t<\infty$.

The remaining of the proof is to Taylor-expand as $t\to0$ (i.e., $n\to\infty$) the
terms $H_{n,p}^t\circ\vartheta_t-H_{n,p}$
and $\log{(1+t\psi_p')}$ appearing
in the right-hand side of~\eqref{eq:expr-fluct} up to the order $O(t^3)$.
After a rather lengthy but not difficult calculation
(see \cite[Section~4]{Bekerman18}),
and using that~$\psi_p$ solves~\eqref{eq:sol-master},
we obtain that
\begin{equation}
	\Enp\ee^{sF_{n,p}}=o(1)
	+\ee^{sm_p+\frac{s^2}2\Sigma^2(h_2)}
	\Esp\exp\Bigl\{-\frac sn\mathsf A_n[\psi_p]+O(nt^2)+O(n^2t^3)\Bigr\}\ind{U_t^n},\label{eq:taylor-expansion}
\end{equation}
where $O(nt^2)$ and $O(n^2t^3)$ are random quantities converging uniformly
to~$0$ as $n\to\infty$ (recall that $t\ceq-\frac{2s}{bn}$),
$m_p\ceq\langle\mu_p,\psi_p'\rangle$ is the limiting mean,
\[
	\Sigma^2(h_2)
\ceq\frac1{2b\pi^2}\int_{-1}^1\int_{-1}^1
\frac{(x+y)^2(1-xy)}{\sqrt{1-x^2}\sqrt{1-y^2}}\,\dd x\dd y\]
is the limiting variance,
and, lastly,
\[\mathsf A_n[\psi_p]\ceq n^2\iint\frac{\psi_p(x)-\psi_p(y)}{x-y}(L_{n,p}-\mu_p)(\dd x)(L_{n,p}-\mu_p)(\dd y)\]
is the so called \textit{anisotropy term}. 
A careful
inspection of the proof of \cite[Proposition~5.4]{Bekerman18} shows that, for~$\lvert s\rvert$ sufficiently
small,
\begin{equation}
	\lim_{n\to\infty}\log\Esp\exp\Bigl\{-\frac sn\mathsf A_n[\psi_p]\Bigr\}=0
\label{eq:anisotropyterm}
\end{equation}
holds provided~$\psi_p$ is of class~$\cC^3$.
This is true as soon as $p>3$; we postpone the proof of this key technical point to Section~\ref{sec:proof-fluct}, see Lemma~\ref{lem:regularity-psi}.
Note that~\eqref{eq:anisotropyterm} is the only place where
the hypothesis $p>3$ is used.
It then follows by the Cauchy-Schwarz inequality that the expectation
in the right-hand side
of~\eqref{eq:taylor-expansion} tends to~$1$.
As the value of $\Sigma^2(h_2)$ simplifies to~$1/4b$ thanks to
Lemma~\ref{lem:variance-fluct}, this
establishes~\eqref{eq:cv-laplace}. 
\end{proof}

We can now give a proof of Theorem~\ref{thm:moment-q}.
\begin{proof}[Proof of Theorem~\ref*{thm:moment-q}]
Recall the earlier computation~\eqref{eq:thm2}
(with $a=1$): as $n\to\infty$,
\begin{equation}\frac{I_q(\UBE)}{I_2(\UBE)}
	=\left(1 - \frac{q-2}{b p n^2} + o\left(\frac1{n^2}\right)\!\right)
	\frac{\left(\Esp{\langle L_{n,p},h_2\rangle^{\sfrac q2}}\right)^{\sfrac1q}}{\left(\Esp{\langle L_{n,p},h_2\rangle}\right)^{\sfrac12}}.\label{eq:first-exp-moment-q}
\end{equation}
We must expand both the numerator and denominator of that latter fraction.
By
Corollary~\ref{cor:troncature}, referring to the truncation function~$\phi$
there, we may replace~$h_2$ by~$h_2\phi$ as this only induces a $o(n^{-2})$ error:
\begin{equation}
	\frac{\Esp{\langle L_{n,p},h_2\rangle^{\sfrac q2}}}{\left(\Esp{\langle L_{n,p},h_2\rangle}\right)^{\sfrac q2}}
=
\frac{\Esp{\langle L_{n,p},h_2\phi\rangle^{\sfrac q2}}+o(\frac1{n^2})}{\left(\Esp{\langle L_{n,p},h_2\phi\rangle}+o(\frac1{n^2})\right)^{\sfrac q2}}.\label{eq:truncatedfraction}
\end{equation}
Applying Taylor-Lagrange's formula to the function $x\mapsto x^{\sfrac q2}$
between
$x_0\ceq\langle\mu_p,h_2\rangle>0$ and $x_0+h$ yields
\begin{align*}
	\left\lvert(x_0+h)^{\frac q2}-x_0^{q/2}-\frac q2h-\frac{q(q-2)}8h^2\right\rvert
	\ind{\{\lvert h\rvert\le\frac{x_0}2\}}&\le C_{q,x_0}\,\lvert h\rvert^3,
\intertext{for some constant~$C_{q,x_0}$
    depending on~$q$ and~$x_0$. Up to enlarging~$C_{q,x_0}$, we also have}
\left\lvert x_0^{q/2}+\frac q2h+\frac{q(q-2)}8h^2\right\rvert
\ind{\{\lvert h\rvert>\frac{x_0}2\}}
&\le C_{q,x_0}\,h^2\ind{\{\lvert h\rvert>\frac{x_0}2\}}.
\end{align*}
We write $\langle L_{n,p},h_2\phi\rangle=x_0+h$ with $h\ceq\frac1nF_{n,p}=\langle L_{n,p},h_2\phi\rangle-\langle\mu_p,h_2\rangle$ and note that
$\lvert h\rvert\le\|h_2\phi\|_\infty+x_0$.
Then, on the one hand, using the above inequalities according to whether $\lvert h\rvert\le x_0/2$
or
$\lvert h\rvert>x_0/2$ gives
\begin{align}
	&n^2\,\left\lvert\Esp{\langle L_{n,p},h_2\phi\rangle^{\frac q2}\ind{\{\lvert F_{n,p}\rvert\le\frac{nx_0}2\}}}
-\langle\mu_p,h_2\rangle^{\frac q2}-\frac q{2n}\Esp{ F_{n,p}}-\frac{q(q-2)}{8n^2}\Esp{F_{n,p}^2}
\right\rvert\notag\\[.4em]
	&\qquad\le C_{q,x_0}\left(\frac{\Esp{\lvert F_{n,p}\rvert^3}}n
	+\esp{F_{n,p}^2\ind{\{\lvert F_{n,p}\rvert>\frac{nx_0}2\}}}\right)\notag\\[.4em]
	&\qquad\le C_{q,x_0}\left(\frac{K^{3/2}}n
	+\left(\|h_2\phi\|_\infty+x_0\right)\Esp{F_{n,p}^2\ind{\{F_{n,p}^2>K\}}}+\esp{F_{n,p}^2\ind{\{\lvert F_{n,p}\rvert>\frac{nx_0}2\}}}\right)\notag\\[.4em]
	&\qquad\le C_{q,x_0}\left(\frac{K^{3/2}}n
	+\left(\|h_2\phi\|_\infty+x_0+1\right)\Esp{F_{n,p}^2\ind{\{F_{n,p}^2>K\}}}\right),\label{eq:hsmall}
\end{align}
for any $0<K\le n^2x_0^2/4$.
On the other hand, for such~$n$ and~$K$,
\begin{equation}
	n^2\Esp{\langle L_{n,p},h_2\phi\rangle^{\frac q2}\ind{\{\lvert F_{n,p}\rvert>\frac{nx_0}2\}}}
\le\frac4{x_0^2}\|h_2\phi\|_\infty^{q/2}\Esp{F_{n,p}^2\ind{\{F_{n,p}^2>K\}}}.
\label{eq:hbig}
\end{equation}
Proposition~\ref{pro:fluct} tells us that the right-hand sides of~\eqref{eq:hsmall}
and~\eqref{eq:hbig} vanish (letting $n\to\infty$ first then $K\to\infty$),
and that $\Var F_{n,p}\to\frac1{4b}$ as $n\to\infty$.
Hence, by triangle inequality
\[\Esp{\langle L_{n,p},h_2\phi\rangle^{\frac q2}}=
\langle\mu_p,h_2\rangle^{\frac q2}+\frac q{2n}\Esp{F_{n,p}}+\frac{q(q-2)}{8n^2}\Esp{F_{n,p}^2}
+o\left(\frac1{n^2}\right).\]
This holds in particular when $q=2$, and going back to~\eqref{eq:truncatedfraction}
we get
\begin{align*}
\frac{\Esp{\langle L_{n,p},h_2\rangle^{\sfrac q2}}}{\left(\Esp{\langle L_{n,p},h_2\rangle}\right)^{\sfrac q2}}
&=
\frac{\Esp{\langle L_{n,p},h_2\phi\rangle^{\sfrac q2}}+o(\frac1{n^2})}{\left(\Esp{\langle L_{n,p},h_2\phi\rangle}+o(\frac1{n^2})\right)^{\sfrac q2}}\\[.4em]
&=\frac{\langle\mu_p,h_2\rangle^{\frac q2}+\frac q{2n}\Esp{F_{n,p}}
+\frac{q(q-2)}{8n^2}\Esp{F_{n,p}^2}+o\left(\frac1{n^2}\right)}{
\langle\mu_p,h_2\rangle^{\frac q2}+\frac q{2n}\Esp{F_{n,p}}+\frac{q(q-2)}{8n^2}(\Esp{F_{n,p}})^2
+o\left(\frac1{n^2}\right)}\\[.4em]
&=1+\frac{q(q-2)}{8\langle\mu_p,h_2\rangle}\cdot\frac{\Var{F_{n,p}}}{ n^2}+o\left(\frac1{n^2}\right)\\[.4em]
&=1+\frac{q(q-2)(p+2)^2}{8bp^2\,n^2}+o\left(\frac1{n^2}\right),
\end{align*}
where we also replaced the value of $\langle\mu_p,h_2\rangle$ computed in
Lemma~\ref{lem:moment-h2}.
Raising this expansion to the power~$1/q$ and returning
to~\eqref{eq:first-exp-moment-q}, we obtain
\[\frac{I_q(\UBE)}{I_2(\UBE)}
=1+\frac{(q-2)(p-2)^2}{16p^2\,d_n}+o\left(\frac1{d_n}\right),\]
because $d_n\sim bn^2/2$ (for $a=1$).
\end{proof}

\subsection{The case $E=\cM_n(\F)$ ($a=2$)}
\label{sec:non-symmetric}
We suppose in this section that $a=2$. We reduce to measures and integrals over~$\RR+$
by performing the change of variables
$y_i\ceq\lvert x_i\rvert^2$ for all $1\le i\le n$.
Specifically, the pushforward of~$\PP_{n,p}$ by the map
$\bs x\in\RR^n\mapsto\bs y\in(\RR+)^n$ is the probability%
\footnote{The normalizing constant~$Z_{n,p}$ remains unchanged; it is just as in~\eqref{eq:normalizing-constant} but with $a=2$.}
\[\wt\PP_{n,p}(\dd\bs y)\ceq\frac1{\wtZ_{n,p}}\:\:
\cdot\!\!\prod_{1\le i<j\le n}\!\!\!\!\left\lvert y_i-y_j\right\rvert^b\cdot\:
\prod_{i=1}^ny_i^{\frac{c-1}2}\cdot\ee^{-2bn\gamma_p
	\|\bs y\|_{p/2}^{p/2}}\,\dd\bs y,
\qquad\bs y\in(\RR+)^n.\]
In particular, for every measurable functions $f\colon\RR+\to\RR$,
and $g\colon\RR\to\RR$,
\begin{equation}
	\int_{(\RR+)^n} g\left(\frac1n\sum_{i=1}^nf(y_i)\!\right)
	\,\wt\PP_{n,p}(\dd\bs y)
	\:=\:
	\int_{\RR^n} g\left(\frac1n\sum_{i=1}^nf\bigl(x_i^2\bigr)\!\right)
	\,\PP_{n,p}(\dd\bs x).\label{eq:changeofvar}
\end{equation}
It is here convenient to work with the empirical probability measure
$\wt L_{n,p}\ceq n^{-1}\sum_{i=1}^n\delta_{y_i}$ where
$\bs y\in(\RR+)^n$ is sampled from~$\wt\PP_{n,p}$, and 
we note that Corollary~\ref{cor:troncature} obviously still applies
with~$\wt L_{n,p}$ in place of~$L_{n,p}$.
Similarly to the previous section, we introduce
\begin{equation*}
	\wt\cI_p(\mu)\ceq4\int_{\RR+}\gamma_p\lvert y\rvert^{p/2}\,\mu(\dd y)
	-\iint_{(\RR+)^2_{\neq}}\log{\lvert x-y\rvert}\,\mu(\dd x)\mu(\dd y).
\end{equation*}
The counterparts of Lemmas~\ref{lem:minimizer} and~\ref{lem:meanconv} are as follows.
\begin{lemma}[Equilibrium measure]
\label{lem:conv-2ndCase}
	There exists a unique element~$\wt\mu_p$ minimizing
	the functional~$\wt\cI_p$ over the space $\cP(\RR+)$ of probability
	measures on~$\RR+$:
	\[\wt\cI_p(\wt\mu_p)\enspace
	=\inf_{\mu\in\cP(\RR+)}\wt\cI_p(\mu).\]
	Furthermore, $\wt\mu_p$ coincides with the image measure of~$\mu_p$ by
	the map $x\mapsto x^2$,
	and it has the following properties:
	\begin{enumerate}[label=(\roman*)]
		\item $\wt\mu_p$ is compactly supported, with
		support~$\cc{0}{1}$;
		\item $\wt\mu_p$ is absolutely continuous with respect to the Lebesgue measure,
		with density
		\[\frac{\dd\wt\mu_p}{\dd y}=\frac{py^{\frac p2-1}}\pi
		\:\int_{\sqrt y}^1
		\frac{u^{-p}}{\sqrt{1-u^2}}\,\dd u,\qquad 0\le y\le1;\]
		\item $\wt\cI_p(\wt\mu_p)=2\log2+\frac3p$.
    \end{enumerate}
\end{lemma}
\begin{proof}
	In the previous section, the minimization
	over all \emph{real} probability measures
	of the functional~$\cI_p$
	corresponding to the weight function
	$w\colon x\mapsto\exp(-\gamma_p\lvert x\rvert^p)$ gave rise to the
	minimizer~$\mu_p$. We are now facing the minimization problem for
	probability measures on $\RR+=\{x^2 :x\in\RR\}$, with the
	weight function being
	$v\colon y\mapsto\exp(-2\gamma_p\,y^{p/2})=w(\sqrt y)^2$. According to
	\cite[Theorem~IV.1.10.(f)]{Saff97},
	the solution of the latter is simply the image measure of~$\mu_p$ by the map
	$x\mapsto x^2$. From this,
	(i)--(iii) easily follow.
\end{proof}
\begin{lemma}[Convergence towards the equilibrium measure]
\label{lem:meanconv2}\leavevmode
\begin{enumerate}[label=(\roman*)]
	\item\label{lem:meanconv2.i} For every continuous, polynomially bounded functions $f,g\colon\RR+\to\RR$,
	it holds that
	\[\lim_{n\to\infty}\EE{g\left(\langle \wt L_{n,p},f\rangle\right)}=g(\langle\wt\mu_p,f\rangle).\]
	\item\label{lem:meanconv2.ii} We have
	\[{\lim_{n\to\infty}{-\frac1{d_n}}\log{Z_{n,p}}}=\log2+\frac{3}{2p}.\]
\end{enumerate}
\end{lemma}
\begin{proof}
We
	apply\footnote{At first sight, the application of Theorem~5.5.1 in~\cite{Hiai00}
	seemingly requires that $\gamma(n)\cequiv\frac{c-1}2$ be nonnegative (which is not
	true if $c=0$). It appears this condition is stated merely for simplicity and,
	referring to the notation there, the proof of the theorem remains valid as long
	as $2\alpha\beta+\gamma\ge0$ and $\gamma>-\alpha$.}
\cite[Theorem~5.5.1]{Hiai00}. This theorem provides the limit of $\wtZ_{n,p}^{\sfrac1{d_n}}$ stated in~(ii), as well as a
large deviation principle with good rate
function $\frac b2\bigl(\wt\cI_p-\wt\cI_p(\wt\mu_p)\bigr)$
for the random probability measures~$\wt L_{n,p}$. By the Borel-Cantelli lemma (see e.g.\ \cite[p.~212]{Hiai00} for details),
this entails that~$\wt L_{n,p}$ converges almost surely
towards~$\wt\mu_p$ in the sense of weak convergence of probability measures,
that is $\langle\wt L_{n,p},f\rangle\to\langle\wt\mu_p,f\rangle$ $\PP$-a.s.\ for
every continuous, bounded function $f\colon\RR+\to\RR$. As before, point~(i)
then follows from the dominated convergence theorem
and Corollary~\ref{cor:troncature}.
\end{proof}

It is now easy to derive a simple proof of
Theorem~\ref{thm:limit-isotropic-constant} in the case $a=2$.
\begin{proof}[Proof of Theorem~\ref*{thm:limit-isotropic-constant}, case $a=2$]
Recalling~\eqref{eq:thm1} and taking into account the change of variable~\eqref{eq:changeofvar}, we have
\[
\frac{1}{\sqrt {d_n}}\,I_q(\UBE)
\sim
\frac{\ee^{-\frac1p-\frac34}}{\sqrt{2\pi}}\,\wtZ_{n,p}^{-1/d_n} \,
\left(\Enp{ \langle\wt L_{n,p},h_1\rangle^{\frac q2}}\right)^{\!\frac1q},
\]
where $h_1(y)\ceq y$.
Lemma~\subref{lem:meanconv2}{ii} tells us that
$\wtZ_{n,p}^{-1/d_n}\sim2\ee^{\frac3{2p}}$, while
Lemma~\subref{lem:meanconv2}{i}
gives
\[\lim_{n\to\infty}\EE{\langle\wt L_{n,p},h_1\rangle^{\frac q2}}
=\langle\wt\mu_p,h_1\rangle^{\frac q2}.\]
Plugging in the value of $\langle\wt\mu_p,h_1\rangle$ given
by Lemma~\ref{lem:moment-h2}, we find
\[
\frac1{\sqrt{d_n}}\,\frac{\left(\esp{{\HS T^q}}\right)^{\sfrac1q}}{\lvert\UBE\rvert^{\sfrac{1}{d_n}}}
\sim\ee^{\frac1{2p}-\frac34}\sqrt{\frac{p}{\pi(p+2)}}.\qedhere
\]

\subsection{Asymptotics of~$c_n$}
It remains to establish the asymptotics~\eqref{eq:asym-c_n} for the
coefficient~$c_n$ involved in Weyl's integration formulas. This is well known from the formulas given in \cite{Anderson10}, see for example \cite[Lemma 3.3]{Kabluchko20c} in the self-adjoint case. We detail the proof for completeness.
\begin{lemma}[Asymptotics of~$c_n$]
\label{lem:cn}
	We have
	\begin{equation*}
		\sqrt n \cdot c_n^{1/d_n} \sim\ee^{\frac34} \sqrt{\frac{4\pi}{ab}}
		\quad 
		\text{as $n\to\infty$.}
	\end{equation*}
\end{lemma}
\begin{proof}
	The coefficient~$c_n$ is related to the volume of the unitary group $U_n(\F)$.
	Specifically, supposing first $a=1$, we know after 
	\cite[Propositions~4.1.1]{Anderson10} that
\[c_n=\frac{\lvert U_n(\F)\rvert}{\lvert U_1(\F)\rvert^n\,n!},\]
	where, according to \cite[Proposition~4.1.14]{Anderson10},
	\[\lvert U_n(\F)\rvert=(2\pi)^{\frac{bn(n+1)}4}\cdot2^{n(1-\frac b2)}\Bigg/\prod_{k=1}^n\Gamma\left(\frac b2\,k\right).\]
	Since $\log{\Gamma(z})=z\log z-z+o(z)$ as $z\to\infty$,
	we have
	\[\log{\Gamma\left(\frac b2k\right)}=\frac b2k\log k+\frac b2\left(\log{\frac b2}-1\right)k+o(k),
	\quad k\to\infty.\]
	Therefore, using that $\sum_{k=1}^nk\log k=\frac12n^2\log n-\frac14n^2+o(n^2)$ and
	$d_n\sim bn^2/2$, we find
	\[\frac1{d_n}\log{\prod_{k=1}^n\Gamma\left(\frac b2\,k\right)}
	=\frac12\log n+\frac12\log{\frac b2}-\frac34+o(1).\]
	Hence
	\begin{equation*}
		\frac1{d_n}\log{c_n}=\frac1{d_n}\log{\lvert U_n(\F)\rvert} + o(1)
		=
		-\frac12\log n+\frac12\log{\frac{4\pi}b}+\frac34+o(1),
	\end{equation*}
	which leads to the statement in the case $a=1$.
	When $a=2$, we have instead $d_n\sim bn^2$
	and, after \cite[Proposition~4.1.3]{Anderson10},
	\[c_n=\frac{\lvert U_n(\F)\rvert^2}{\lvert U_1(\F)\rvert^n\,n!}\,
	2^{-\frac{bn(n-1)}2}.\]
	So the difference with above is that, here,
	\begin{align*}
		\frac1{d_n}\log{c_n}&=\frac2{d_n}\log{\lvert U_n(\F)\rvert}
		-\frac12\log2 + o(1)\\
		&=
		-\frac12\log n+\frac12\log{\frac{2\pi}b}+\frac34+o(1).
	\end{align*}
	Hence the statement in the case $a=2$.
\end{proof}

\begin{remark}
	We can recover the asymptotic volumes of the Schatten unit balls. These were recently derived in~\cite[Theorem~3.1]{Kabluchko20a}
	and~\cite[Theorem~1]{Kabluchko20b},
	completing the much earlier computations of
	Saint-Raymond~\cite{SaintRaymond84}. Indeed, starting from the expression of the
	volume of $\UBE$ in~\eqref{eq:expr-volume} and plugging in the asymptotics
	of~$Z_{n,p}$ (Lemmas~\ref{lem:meanconv} and~\ref{lem:meanconv2}) and of~$c_n$
	(Lemma~\ref{lem:cn}), we find
	\[
	\left\lvert{\UBE}\right\rvert^{\frac1{d_n}}
	=
	\left(
	ab
	n\gamma_p\right)
	^{\frac1p}\cdot
	\frac{c_n^{\sfrac1{d_n}}}{\Gamma\left(1+\frac{d_n}p\right)^{\!\sfrac1{d_n}}}\cdot Z_{n,p}^{1/d_n}
	\sim(abn\gamma_p)^{\frac1p}\cdot\frac{\ee^{\sfrac34}\sqrt{\frac{4\pi}{nab}}}{\left(\frac{d_n}{{\ee}p}\right)^{\!\sfrac1p}}\cdot\frac12\ee^{-\frac3{2p}}.\]
	Hence, because $d_n\sim abn^2/2$,
	\begin{flalign*}
	&&\hphantom{\qed}\lim_{n\to\infty}n^{\frac12+\frac1p}\,
	\lvert\UBE\rvert^{\frac1{d_n}}=
	(2p\gamma_p)^{\frac1p}\ee^{\frac34-\frac1{2p}}
	\sqrt{\frac\pi{ab}}.&&\qed
    \end{flalign*}
\end{remark}

\end{proof}

\section{Properties of the equilibrium measures}\label{sec:properties}
In this section we gather some auxiliary results on the equilibrium
measure~$\mu_p$ of Lemma~\ref{lem:minimizer}.
We first carry out some easy computations. Next, we establish the regularity
of the solution to the master equation occurring in the proof of
Proposition~\ref{pro:fluct}.

\subsection{Some integral computations}
We recall the notation $h_2(x)\ceq x^2$, $h_1(y)\ceq y$,
and the equilibrium measures~$\mu_p$ and~$\wt\mu_p$,
see Lemmas~\ref{lem:minimizer} and~\ref{lem:conv-2ndCase}.
We further denote by $\varrho(\dd x)\ceq(\pi\sqrt{1-x^2})^{-1}\ind{\{\lvert x\rvert<1\}}$ the standard Arcsine distribution on $\cc{-1}{1}$.
\begin{lemma}\label{lem:moment-h2}
	We have
	\[\langle\mu_p,h_2\rangle=\langle\wt\mu_p,h_1\rangle=\frac p{2p+4}.\]
\end{lemma}
\begin{proof}
	First,
	$\langle\mu_p,h_2\rangle=\langle\wt\mu_p,h_1\rangle$ because~$\wt\mu_p$ is
	the image measure of~$\mu_p$ by~$h_2$. Next we know
	\cite[Lemma~4.1]{Assche20} that~$\mu_p$ is the
	distribution of~$AB$ where~$A$ and~$B$ are independent variables
	with~$A$ following~$\varrho$ and~$B$
	following the $\mathrm{Beta}(p,1)$ distribution. We easily conclude that
	\[
	\langle\mu_p,h_2\rangle=\esp{A^2}\cdot\esp{B^2}=\frac12\cdot\frac p{p+2}.
	\qedhere
	\]
\end{proof}
\begin{lemma}\label{lem:variance-fluct}
	We have
\[
\frac1{\pi^2}\int_{-1}^1\int_{-1}^1
\frac{(x+y)^2(1-xy)}{\sqrt{1-x^2}\sqrt{1-y^2}}\,\dd x\dd y
=\frac12.\]
\end{lemma}
\begin{proof}
	Let $X,Y$ be independent variables with law~$\varrho$. Then the left-hand
	side is
	\begin{align*}
		\esp{(X+Y)^2(1-XY)}
		&=\esp{(X+Y)^2}-2\esp{X^2}\esp{Y^2}-\esp{X^3}\esp{Y}
		-\esp{X}\esp{Y^3}\\[.4em]
		&=\frac1{2},
	\end{align*}
	using that $\esp X=\esp Y=0$ and $\esp{X^2}=\esp{Y^2}=\frac12$.
\end{proof}

\subsection{Regularity of the solution to the master equation}
\label{sec:proof-fluct}
We establish the regularity of the solution~$\psi_p$ to the
master equation
\begin{equation}
	\Xi_p\psi=\frac12h_2+c_{h_2},\label{eq:functeq}
\end{equation}
where
\[\Xi_p\psi(x)\ceq-\frac12
V'_p(x)\psi(x)+\int\frac{\psi(x)-\psi(y)}{x-y}\,\mu_p(\dd y),
\qquad x\in\RR.\]
A general expression  in terms of the test function
and of the equilibrium measure  is given  in \cite[Lemma~3.3]{Bekerman18}, see
also Section~B.5 there. The inversion of the master operator was first achieved
in~\cite[Lemma~3.2]{Bekerman15}.
In our framework, the expression of~$\psi_p$ can be made quite explicit: first, we
compute
$c_{h_2}=-1/4$, and
\begin{equation}
\label{eq:psiin}
\psi_p(x) =  \begin{cases}
             -\dfrac{x\sqrt{1-x^2}}{2\pi f_p(x)},
               &\text{if }\lvert x\rvert\le1, 
               \\[3ex]
             -\dfrac{\lvert x\rvert\sqrt{x^2-1}}{2\zeta'_p(x)},
	           &\text{if }\lvert x\rvert>1,
	       \end{cases}
\end{equation}
where
\begin{equation}
	f_p(x)\ceq\frac{px^{p-1}}\pi\int_x^1\frac{v^{-p}}{\sqrt{1-v^2}}\,\dd v,\label{eq:exprfp}
\end{equation}
is the Lebesgue density of~$\mu_p$ which we have already seen in
Lemma~\subref{lem:minimizer}{ii}, and~$\zeta'_p$ is the odd function given for $x>1$ by
\begin{align}
	\zeta'_p(x)&\ceq\frac{\dd}{\dd x}\left(-{\int\log{\lvert x-y\rvert}\,\mu_p(\dd y})+\frac12\,V_p(x)\right)\notag\\[.4em]
	&\colspace= px^{p-1}\int_1^x\frac{v^{-p}}{\sqrt{v^2-1}}\,\dd v,\label{eq:zetaprime}
\end{align}
see \cite[pp.\ 240-241]{Saff97}. We further observe the similarity
between~\eqref{eq:exprfp} and~\eqref{eq:zetaprime}: applying the change of
variable $t\gets(1-v)/x$, we have, for every $0<x<1$,
\begin{align*}
	\pi f_p(1-x)&=\frac{p(1-x)^{p-1}}\pi\int_{1-x}^1\frac{u^{-p}}{\sqrt{1-u^2}}\,\dd u\\[.4em]
	&=p(1-x)^{p-1}\sqrt x\int_0^1\frac{(1-xt)^{-p}}{\sqrt{2-xt}}\,\frac{\dd t}{\sqrt t}\\[.4em]
	&=\sqrt x\,w_p(-x),
\end{align*}
where we have set
\begin{equation}
	w_p(u)\ceq p(1+u)^{p-1}\int_0^1\frac{(1+tu)^{-p}}{\sqrt{2+tu}}\,\frac{\dd t}{\sqrt t},\qquad u\ge-1.\label{eq:exprw}
\end{equation}
Similarly, we can see that $\zeta'_p(1+x)=\sqrt x\,w_p(x)$ holds for all $x>0$.
Because~$\psi_p$ is odd, we conclude that the solution $\psi_p\in\cC^1(\RR,\RR)$ to
the equation~\eqref{eq:functeq}
is more simply expressed by
\begin{equation}
	\psi_p(x)=-\frac{x\sqrt{1+\lvert x\rvert}}{2w_p(\lvert x\rvert-1)},\qquad x\in\RR.
	\label{eq:exprpsi2}
\end{equation}

We denote by~$\lceil p\rceil$ the smallest integer greater than or equal to~$p$.
\begin{lemma}[Regularity of~$\psi_p$]
\label{lem:regularity-psi}
	For every $p\in\oo{1}{\infty}$,
	the function~$\psi_p$ is of class~$\cC^{\lceil p\rceil-1}$. In~particular,
	it is of class~$\cC^3$ when $p>3$.
\end{lemma}
\begin{proof}
    Since~$\psi_p$ is odd, it suffices to check the regularity
    on~$\co{0}{\infty}$. It is plain that the function~$w_p$ in~\eqref{eq:exprw}
    does not vanish on~$\oo{-1}{\infty}$ and,
    by differentiation under the integral sign, that it is of class~$\cC^\infty$
    there. Given~\eqref{eq:exprpsi2}, it is then clear that~$\psi_p$ is
    of class~$\cC^\infty$ on~$\oo{0}{\infty}$.
    It remains to check that~$\psi_p$ has $\cC^{\lceil p\rceil -1}$-regularity
    at~$0^+$.

 Because of the expression~\eqref{eq:psiin} and of
 $x\mapsto\sqrt{1-x^2}$ being smooth at~$0$, the regularity of~$\psi_p$ at~$0$ is
	equivalent to that of
	$\bar\psi_p(x)\ceq x/f_p(x)$.
For $0<x<y<1$, the Taylor series
	\[(1-v^2)^{-\frac12}=\sum_{k=0}^\infty\frac{\binom{2n}n}{4^n}\,v^{2n}\]
	is uniformly convergent on~$\cc{x}{y}$, and expanding in~\eqref{eq:exprfp} gives
	\[f_p(x)=\frac{px^{p-1}}\pi\mathop{\mathrm{lim}\!\!\uparrow}_{y\uparrow1}\int_x^y\frac{v^{-p}}{\sqrt{1-v^2}}\,\dd v
	=\frac p\pi\sum_{n=0}^\infty\frac{\binom{2n}n}{4^n}\,x^{p-1}\int_x^1v^{2n-p}\,\dd v,\]
	by the monotone convergence theorem. The latter integral equals~$-{\log x}$ if
	$p=2n+1$, and $(x^{2n+1-p}-1)/(p-2n-1)$ otherwise.
In any case, there exist
    constants $A,B\in\RR$ (possibly zero) such that
	\[g(x)\ceq f_p(x)+Ax^{p-1}\log x+Bx^{p-1}=\frac p\pi\!\sum_{\substack{n\ge0\\2n+1\neq p}}\!\!\!
	\frac{\binom{2n}n}{4^n(p-2n-1)}\,x^{2n}\]
	is decomposed as a power series with radius~$1$, so~$g$ defines a function of class~$\cC^\infty$
	at~$0$. Because, for $p>1$, $Ax^{p-1}\log x+Bx^{p-1}$ has
	$\cC^{\lceil p\rceil -2}$-regularity at~$0^+$, $f_p$ is therefore of
	class~$\cC^{\lceil p\rceil-2}$ at~$0^+$. Then so is~$1/f_p$ since
	$f_p(0)=g(0)=p\pi/(p-1)\neq0$, and
	we conclude by Leibniz formula that~$\bar\psi_p$ (and thus~$\psi_p$) is of class~$\cC^{\lceil p\rceil-1}$ at~$0^+$.
\end{proof}

\printbibliography

\vspace{1em}\small
\noindent
\textsc{Benjamin Dadoun, Matthieu Fradelizi, Olivier Guédon, Pierre-André Zitt}\\
\address{\univaddress}\\
	\mail{benjamin.dadoun@univ-eiffel.fr}\\ 
	\mail{matthieu.fradelizi@univ-eiffel.fr}\\
	\mail{olivier.guedon@univ-eiffel.fr}\\
	\mail{pierre-andre.zitt@univ-eiffel.fr}

\end{document}